\documentclass[11pt]{amsart}
\usepackage[babel]{csquotes}
\usepackage{enumitem}
\usepackage{amsmath,amsthm,amssymb,mathrsfs,amsfonts,verbatim,enumitem,color,leftidx}
\usepackage{mathabx}
\usepackage[left=2.9cm,right=2.9cm,top=3.3cm,bottom=3.4cm,a4paper]{geometry}
\usepackage{mathtools}
\usepackage{lipsum}
\usepackage{etoolbox} 
\usepackage{bbm}
\usepackage[all,tips]{xy}
\usepackage{here}
\usepackage{graphicx,ifpdf}
\usepackage{stmaryrd}
\ifpdf
\fi

\usepackage[right,displaymath,mathlines]{lineno}

\definecolor{darkolivegreen}{rgb}{0.33, 0.42, 0.18}
\definecolor{cadmiumgreen}{rgb}{0.0, 0.42, 0.24}
\definecolor{calpolypomonagreen}{rgb}{0.12, 0.3, 0.17}
\usepackage[colorlinks]{hyperref}
\hypersetup{
linkcolor=blue,          
citecolor=cadmiumgreen,      
}

\usepackage{tikz}
\usetikzlibrary{arrows,snakes,backgrounds}
\usetikzlibrary{arrows.meta,calc,patterns,positioning}

\allowdisplaybreaks[4]

\newtheorem{thm}{Theorem}[section]

\newtheorem{lem}[thm]{Lemma}
\newtheorem{lemma}[thm]{Lemma}
\newtheorem{coro}[thm]{Corollary}
\newtheorem{corollary}[thm]{Corollary}
\newtheorem{prop}[thm]{Proposition}
\newtheorem{proposition}[thm]{Proposition}

\theoremstyle{definition}

\newtheorem{definition}[thm]{Definition}

\newtheorem{remark}[thm]{Remark}

\theoremstyle{remark}

\numberwithin{equation}{section}

\definecolor{esperance}{rgb}{0.0,0.5,0.0}

\usetikzlibrary{arrows,snakes,backgrounds}

\newcommand{\PGL}{\operatorname{PGL}}

\newcommand\diag[1]{\operatorname{diag}\left(#1\right)}

\newcommand{\onto}{\xymatrix{\ar@{>>}[r]&}}

\begin{document}

\title[Chamber zeta function of $\operatorname{PGL}_3$]{Chamber zeta function and closed galleries in the standard non-uniform complex from $\operatorname{PGL}_3$}
\author{Soonki Hong}
\address{Department of Mathematical Education\\Catholic Kwandong University\\Beomil-ro, Gangneung 25601, Republic of Korea}
\email{soonki.hong@cku.ac.kr}
\author{Sanghoon Kwon}
\address{Department of Mathematical Education\\Catholic Kwandong University\\Beomil-ro, Gangneung 25601, Republic of Korea}
\email{skwon@cku.ac.kr, shkwon1988@gmail.com}

\maketitle

\begin{abstract}
We introduce the \emph{chamber zeta function} for a complex of groups, defined via an Euler product over primitive tailless chamber galleries, extending the Ihara--Bass framework from weighted graphs to higher-rank settings. Let $\mathcal{B}$ be the Bruhat--Tits building of $\mathrm{PGL}_{3}(F)$ for a non-archimedean local field $F$ with residue field $\mathbb{F}_{q}$. For the standard arithmetic quotient $\Gamma\backslash\mathcal{B}$ with $\Gamma=\mathrm{PGL}_{3}(\mathbb{F}_{q}[t])$, we prove an Ihara--Bass type \emph{determinant formula} expressing the chamber zeta function as the reciprocal of a characteristic polynomial of a naturally defined chamber transfer operator. In particular, the chamber zeta function is \emph{rational} in its complex parameter. As an application of the determinant formula, we obtain explicit counting results for closed gallery classes arising from tailless galleries in $\mathcal{B}$, including exact identities and spectral asymptotics governed by the chamber operator.
\end{abstract}

\tableofcontents

\section{Introduction}

The Ihara--Bass zeta function for a finite regular graph expresses the Euler product over primitive closed geodesics as the reciprocal of a determinant attached to an adjacency operator. Starting from Ihara's original work on rank-one $p$-adic groups \cite{I} and its graph-theoretic refinements by Sunada \cite{Su}, Hashimoto \cite{Has}, and Bass \cite{B}, a broad ``zeta--determinant'' paradigm has developed at the interface of number theory, representation theory, and discrete geometry.
In recent years, several extensions to infinite graphs and weighted settings have been investigated; in particular, Deitmar--Kang \cite{DK18} zeta functions built from geodesics in the universal covering tree of infinite weighted graphs. In \cite{HK2}, the authors also explored the Selberg zeta function of graphs of groups and compared it with the zeta function introduced by Deitmar and Kang. 
On the higher-rank side, Kang--Li \cite{KL} and Kang--Li--Wang \cite{KLW10,KLW18} studied edge and chamber zeta functions for finite (cocompact) quotients of the $\widetilde{A}_2$ Bruhat--Tits building and established determinant formulas with applications to Ramanujan complexes.

\subsection*{Goal and context}
This paper addresses the higher-rank, non-uniform corner that naturally remains open in the above picture.
Let $\mathcal{B}$ be the Bruhat--Tits building of $\mathrm{PGL}_3(F)$ for a non-archimedean local field $F$ with residue field $\mathbb{F}_q$.
We focus on the \emph{standard non-uniform arithmetic quotient}
\[
G=\mathrm{PGL}_3\bigl(\mathbb{F}_q(\!(t^{-1})\!)\bigr),
\qquad
\Gamma=\mathrm{PGL}_3(\mathbb{F}_q[t]),
\qquad
X=\Gamma\backslash \mathcal{B},
\]
which is a non-compact complex of finite volume and admits a combinatorial cusp geometry.
Our aim is to define and compute a \emph{chamber-level} zeta function for $X$ and to prove an Ihara--Bass type determinant formula in this non-uniform setting.

\begin{figure}[ht]
\centering
\begin{tikzpicture}[>=latex, node distance=3.4cm, font=\small]

\node at (-4.8,3.6) {quotient type};
\node[rotate=90] at (-4.95,1.7) {non-uniform};
\node[rotate=90] at (-4.9,-1.8) {finite / cocompact};

\node at (-2.0,-3.6) {edge-level};
\node at (3.7,-3.6) {chamber-level};

\draw[dashed] (0.77,4) -- (0.77,-4);

\draw[dashed] (-5.2,0) -- (6.2,0);

\node[anchor=south] at (-2,3.3) {edge zeta};
\node[anchor=south] at (3.4,3.3) {chamber zeta};

\node[draw, rounded corners, align=center, minimum width=3.2cm, minimum height=1.6cm] (Q1) at (-2,-1.8)
{finite graphs,\\
(regular or irregular)\\
(Ihara, Bass, Hashimoto\\
Sunada, etc.)};

\node[draw, rounded corners, align=center, minimum width=3.2cm, minimum height=1.6cm] (Q2) at (3.5,-1.8)
{$\widetilde{A}_2$-buildings,\\
cocompact lattices in\\
$\mathrm{PGL}_3(F)$, chamber zeta\\
(\cite{KL},\cite{KLW10},\cite{KLW18})};

\node[draw, rounded corners, align=center, minimum width=3.2cm, minimum height=1.6cm] (Q3) at (-2,1.8)
{weighted graphs (\cite{DK18}),\\standard non-uniform\\
complex for $\mathrm{PGL}_3$\\
 (\cite{HK3})
};

\node[draw, rounded corners, thick, align=center, minimum width=3.2cm, minimum height=1.6cm, fill=gray!10] (Q4) at (3.5,1.8)
{standard non-uniform\\
complex for $\mathrm{PGL}_3$,\\
type~1 chamber zeta\\
\textbf{(this paper)}};

\draw[->, thick] (Q3.east) -- (Q4.west);
\draw[->, thick] (Q2.north) -- (Q4.south);

\end{tikzpicture}
\caption{A schematic position of this work among edge and chamber zeta functions for cocompact and non-uniform quotients.}
\end{figure}
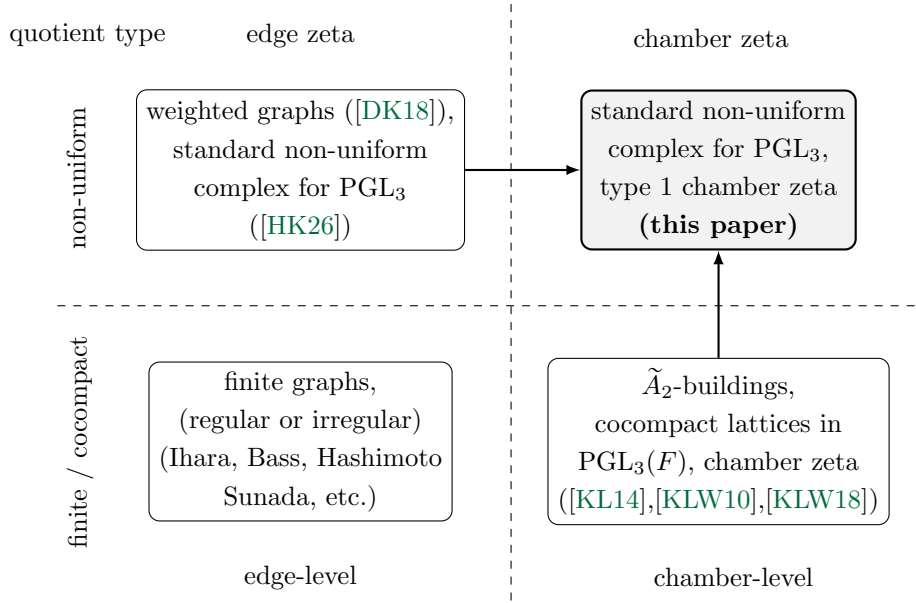

\medskip

\paragraph{Notation.}
Let $\pi:\mathcal{B}\to X=\Gamma\backslash\mathcal{B}$ be the quotient map.
Throughout the paper, when working in the building $\mathcal{B}$ \emph{before introducing the quotient},
we may denote unpointed chambers in $\mathcal{B}$ by $C,D$ (and similarly $C_0,\dots,C_{n-1},C_n$ for a gallery);
once $X$ is introduced, we reserve $C,D$ for unpointed chambers in $X$ and write their lifts in $\mathcal{B}$ as
$\widetilde{C},\widetilde{D}$.
Pointed chambers are denoted by $c,d$ in $X$ and by $\tilde c,\tilde d$ in $\mathcal{B}$.
An unpointed gallery is denoted by $\mathbf{C}=(C_0,\dots,C_{n-1},C_n)$, and a pointed gallery by
$\mathbf{c}=(c_0,\dots,c_{n-1},c_n)$.
For a pointed chamber $c$ we write $|c|$ for its underlying unpointed chamber, and for a pointed gallery
$\mathbf{c}$ we set $|\mathbf{c}|=(|c_0|,\ldots,|c_{n-1}|,|c_n|)$.
Square brackets $[\mathbf{c}]$ denote the shift-equivalence class of a closed pointed gallery.

\subsection*{Why weights are intrinsic in the non-uniform case}
In a cocompact quotient, a local adjacency step of pointed chambers lifts uniquely (up to the group action), so the chamber zeta function can be defined as an unweighted Euler product.
In contrast, in the non-uniform quotient $X=\Gamma\backslash\mathcal{B}$, a single adjacency step in the quotient may admit \emph{several distinct tailless lifts} in the building.
Consequently, any Euler product intended to enumerate tailless chamber galleries must record this lift multiplicity.
We incorporate it through a weight function $w(c,c')$ on adjacent type~1 pointed chambers $c,c'$ in $X$, defined as the number of tailless lifts in $\mathcal{B}$ of the step $c\to c'$ (see Section~\ref{sec:3}).
For a closed type~1 gallery $\mathbf{c}=(c_0,\dots,c_{n-1},c_n)$ in $X$ (which means $c_n=c_0$), we set
\[
w(\mathbf{c})=\prod_{j \bmod n} w(c_j,c_{j+1}),
\]
and for a closed gallery class $[\mathbf{c}]$ we write $w([\mathbf{c}])=w(\mathbf{c})$. We also denote by $\ell([\mathbf{c}])=n$ the length of the class $[\mathbf{c}]$.

\medskip

\begin{definition}[Type 1 chamber zeta function]
The \emph{type~1 chamber zeta function} of $X$ is defined by the Euler product
\[
Z_\Gamma(u)
=
\prod_{[\mathbf{p}]}
\left(1-w([\mathbf{p}])\,u^{\ell([\mathbf{p}])}\right)^{-1},
\]
where the product runs over all primitive type~1 closed gallery classes in $X$.
(We recall the notions of type, taillessness, closedness, and primitivity in Section~\ref{sec:pgl3-building-pointed}.)
\end{definition}
\medskip

\noindent\textbf{Main results.}
Our first result is an explicit computation of $Z_\Gamma(u)$ for the standard non-uniform quotient.

\begin{thm}\label{thm:1.1}
Let $G=\mathrm{PGL}_3\bigl(\mathbb{F}_q(\!(t^{-1})\!)\bigr)$ and $\Gamma=\mathrm{PGL}_3(\mathbb{F}_q[t])$.
Then the type~1 chamber zeta function $Z_\Gamma(u)$ converges for $|u|$ sufficiently small and is given by
\[
Z_{\Gamma}(u)
=
\frac{(1-q^4u^6)(1-q^2u^3)}{(1-q^3u^6)(1-q^3u^3)}.
\]
In particular, $Z_\Gamma(u)$ is a rational function of $u$.
\end{thm}

The second result is the determinant formula that underlies Theorem~\ref{thm:1.1}.
Let $C(X)$ be the set of type~1 pointed chambers in $X$ and let
\[
S(C(X))=\bigoplus_{c\in C(X)} \mathbb{C}c
\]
be the formal vector space with basis $C(X)$.
Define the \emph{weighted chamber transfer operator} $T\colon S(C(X))\to S(C(X))$ by
\[
T(c)=\sum_{c'} w(c,c')\,c',
\]
where the sum runs over all type~1 pointed chambers $c'$ edge-adjacent to $c$ in $X$.
Although $X$ is infinite, we show that the powers $T^n$ are \emph{traceable} in the sense that the diagonal coefficient sum
\[
\mathrm{Tr}(T^n)=\sum_{c\in C(X)} \langle T^n c, c\rangle
\]
is finite for each fixed $n$, and that the trace coincides with the weighted count of closed gallery classes:
\[
\mathrm{Tr}(T^n)
=
\sum_{[\mathbf{c}]:\,\ell([\mathbf{c}])=n} w([\mathbf{c}])\,\ell([\mathbf{c}]_0),
\]
where $[\mathbf{c}]_0$ denotes the primitive class underlying $[\mathbf{c}]$.
As a consequence, the logarithmic derivative of $Z_\Gamma(u)$ is encoded by the traces of $T^n$.

To formulate the determinant identity without appealing to operator-spectrum language, we define
$\det(I-uT)$ by truncation.
More precisely, we introduce finite truncations $T_{\le R}$ of $T$ supported on a finite subcomplex obtained by cutting off $X$ along cusp depth, and we set
\[
\det(I-uT)
\]
to be the unique power series whose coefficients stabilize to those of 
\[\det(I-uT_{\le R})\quad\text{ as }\quad R\to\infty.\]
The traceability mechanism implies this coefficient stabilization.
With this convention, we obtain the Ihara--Bass type determinant formula
\[
Z_\Gamma(u)=\frac{1}{\det(I-uT)}
\qquad\]
as an identity of formal power series, hence of rational functions.

\medskip

\noindent\textbf{Counting closed galleries.}
Let
\[
N_n(X)=\sum_{[\mathbf{c}]:\,\ell([\mathbf{c}])=n} w([\mathbf{c}])\,\ell([\mathbf{c}]_0)
\]
be the weighted number of type~1 closed gallery classes of length $n$ in $X$.
Then
\[
Z_\Gamma(u)
=
\exp\left(\sum_{n\ge 1}\frac{N_n(X)}{n}u^n\right),
\]
and Theorem~\ref{thm:1.1} yields explicit formulas for $N_n(X)$.

\begin{corollary}\label{coro:intro_count}
Let $\Gamma=\mathrm{PGL}_3(\mathbb{F}_q[t])$ and $X=\Gamma\backslash\mathcal{B}$.
Then $N_n(X)=0$ unless $3|n$, and for $n=3r$ one has
\[
N_{n}(X)=
\begin{cases}
3q^{3r}-3q^{2r}, & \text{if } n\equiv 3 \pmod{6},\\
3q^{3r}-9q^{2r}+6q^{\frac{3}{2}r}, & \text{if } n\equiv 0 \pmod{6}.
\end{cases}
\]
\end{corollary}

\medskip

\noindent\textbf{Idea of the method: monotonicity, traceability, truncation.}
The key technical point is a cusp monotonicity phenomenon in the standard non-uniform quotient $X$.
Certain directed adjacency steps force a strictly increasing cusp depth (or, equivalently, a strictly decreasing depth in the opposite direction).
Once such a step occurs inside a tailless gallery, returning to the starting chamber requires paying a uniform length cost.
This yields a strong local finiteness statement: for each fixed length $n$, there are only finitely many admissible closed gallery classes of length $n$ in $X$.
This finiteness implies traceability of $T^n$ and justifies the trace--gallery correspondence.
The determinant computation then proceeds by a Fubini-type rearrangement together with a cusp truncation that reduces $\det(I-uT_{\le R})$ to a finite matrix determinant, which can be computed explicitly in terms of $q$.

\subsection*{Generality and outlook}

Although our main results are proved for the standard non-uniform quotient
$\Gamma\backslash\mathcal{B}$ with $\Gamma=\mathrm{PGL}_{3}(\mathbb{F}_{q}[t])$,
we view this paper primarily as a \emph{prototype} for a higher-rank zeta program
in finite-volume (non-cocompact) settings.
The point is that the chamber zeta function is not an ad hoc gadget:
it is simultaneously
(i) an Euler product over primitive tailless chamber galleries,
(ii) a dynamical zeta function attached to a natural ``chamber transfer''
(Hecke) operator,
and (iii) a determinant-type invariant whose coefficients are governed by traces
of iterates of that operator.

\medskip

A major novelty in the finite-volume higher-rank situation is that one cannot
directly take determinants on the full quotient complex, since the chamber set
is infinite and the transfer operator is not represented by a finite matrix.
Our approach resolves this by a robust truncation--stabilization mechanism:
we compute explicit characteristic polynomials on finite truncations,
prove coefficientwise stabilization, and then pass to the limit.
This provides a clean bridge from combinatorial-geometric input (tailless galleries)
to spectral output (a rational zeta function with an Ihara--Bass type formula),
in a context where such a bridge is not automatic.

\medskip

We expect that the same mechanism extends beyond the present case along several
natural directions:

\begin{itemize}
\item \textbf{Other groups and higher rank.}
For affine buildings of $\mathrm{PGL}_{d}$ (and more generally split reductive
groups over local fields), one can define analogous gallery-based zeta functions
for top-dimensional cells, and the associated transfer operators belong to the
appropriate Iwahori--Hecke algebra.
The present paper isolates the finite-volume difficulties and demonstrates
how a determinant formula can still be recovered via stabilized truncations.

\item \textbf{Other arithmetic lattices and congruence quotients.}
Replacing $\mathrm{PGL}_{3}(\mathbb{F}_{q}[t])$ by congruence subgroups or by
other $S$-arithmetic lattices should lead to zeta functions encoding finer
arithmetic and representation-theoretic information.
In particular, the poles and zeros predicted by the determinant expression
should be governed by Hecke eigenvalues on spaces of $\Gamma$-invariant functions
(or more generally on (co)homology with local coefficients).

\item \textbf{Face/panel variants and colored gallery constraints.}
One may define zeta functions that count closed galleries with additional
constraints (panel types, colorings, orientations, or local weights),
leading to families of transfer operators and potential factorizations
reflecting finer geometric decompositions.
\end{itemize}

\medskip

From this perspective, our explicit computation for
$\Gamma=\mathrm{PGL}_{3}(\mathbb{F}_{q}[t])$ should be regarded as the first
fully worked out finite-volume higher-rank example in which
\emph{both} the Euler product definition and the determinant/spectral side are
completely accessible.
We hope that this example will serve as a concrete testing ground for broader
questions on higher-rank zeta functions, Hecke dynamics on buildings, and their
interactions with automorphic spectra.

\medskip

\noindent\textbf{Organization of the paper.}
We review the chamber-system formalism and the Weyl-distance axioms for $(W,S)$-buildings in Section~\ref{sec:chamber-system-building}. In Section~\ref{sec:pgl3-building-pointed}, we specialize the abstract framework to the $\widetilde{A}_2$-type Bruhat--Tits building of $\operatorname{PGL}_3(F)$ and present its simplicial (lattice) realization.
Section~\ref{sec:3} constructs an explicit fundamental domain for $\Gamma\backslash\mathcal{B}$, and defines the admissible (type~$1$) pointed galleries together with the associated weight function on $X$. In Section~\ref{sec:4} we introduce the type~$1$ chamber zeta function and prove the Ihara--Bass type determinant formula by combining trace identities with a finite-depth truncation argument. Section~\ref{sec:5} computes $\det(I-uT)$ explicitly and completes the proofs of Theorem~\ref{thm:1.1} and Corollary~\ref{coro:intro_count}.

\subsubsection*{Acknowledgement} This research was supported by Basic Science Research Program through the National Research Foundation of Korea(NRF) funded by the Ministry of Education(No. RS-2025-25415913) and (No. RS-2023-00237811).

%
\section{Chamber systems and \texorpdfstring{$(W,S)$}{(W,S)}-buildings}\label{sec:chamber-system-building}

This section fixes notation and recalls the building language in the chamber-system formalism.
We emphasize the Weyl-distance axioms because later constructions (pointed chambers, directed transitions,
and trace formulas) ultimately reduce to counting certain gallery classes.
Standard references include \cite{AbramenkoBrown} and \cite{Ronan}.

\subsection{Chamber systems, panels, and galleries}\label{subsec:chamber-system}

Let $S$ be a finite set.

\begin{definition}[Chamber system]\label{def:chamber-system}
An \emph{$S$-chamber system} is a set $\mathcal{C}$ (whose elements are called \emph{chambers})
equipped with equivalence relations $\sim_s$ on $\mathcal{C}$ indexed by $s\in S$.
If $C\sim_s D$ we say that $C$ and $D$ are \emph{$s$-adjacent}.
The equivalence classes of $\sim_s$ are called \emph{$s$-panels}.
\end{definition}

\begin{definition}[Galleries]\label{def:galleries}
A \emph{gallery} of length $n$ is a sequence of chambers
\[
\mathbf{C}=(C_0,C_1,\ldots,C_{n})
\]
such that for each $i$ there exists $s_i\in S$ with $C_{i-1}\sim_{s_i} C_i$.
The word $s_1\cdots s_n$ is called the \emph{type} of the gallery.
A gallery is \emph{minimal} if no gallery connecting $C_0$ to $C_n$ has smaller length.
\end{definition}

\begin{definition}[Thickness and local finiteness]\label{def:thickness-localfiniteness}
An $S$-chamber system $\mathcal{C}$ is \emph{thick} if every $s$-panel contains at least three chambers.
It is \emph{locally finite} if every $s$-panel contains only finitely many chambers.
\end{definition}

\subsection{Coxeter systems and Coxeter complexes}\label{subsec:coxeter}

\begin{definition}[Coxeter system]\label{def:coxeter-system-abstract}
A \emph{Coxeter system} is a pair $(W,S)$ consisting of a group $W$ generated by $S$
subject to relations $(st)^{m(s,t)}=1$, where $m(s,s)=1$ and
$m(s,t)\in\{2,3,\dots\}\cup\{\infty\}$ for $s\neq t$.
Let $\ell:W\to \mathbb{Z}_{\ge 0}$ denote the associated length function.
A word $s_1\cdots s_n$ in $S$ is \emph{reduced} if $n=\ell(s_1\cdots s_n)$.
\end{definition}

\begin{definition}[Coxeter complex]\label{def:coxeter-complex-abstract}
Let $(W,S)$ be a Coxeter system.
The \emph{Coxeter complex} $\Sigma(W,S)$ is the $S$-chamber system with chamber set $W$,
where $w\sim_s ws$ for each $s\in S$.
It comes with a canonical Weyl distance $\delta_\Sigma(w_1,w_2)=w_1^{-1}w_2$.
\end{definition}

\subsection{\texorpdfstring{$(W,S)$}{(W,S)}-buildings via Weyl distance}\label{subsec:building-weyl-distance}

\begin{definition}[Weyl distance]\label{def:weyl-distance}
Let $(W,S)$ be a Coxeter system and let $\mathcal{C}$ be an $S$-chamber system.
A map $\delta\colon\mathcal{C}\times\mathcal{C}\to W$ is called a \emph{Weyl distance} if it satisfies:
\begin{enumerate}[label=(WD\arabic*), leftmargin=15mm]
\item \textbf{(Identity)} $\delta(C,D)=1$ if and only if $C=D$.
\item \textbf{(One-step rule)} If $\delta(C,D)=w$ and $D\sim_s D'$, then
$\delta(C,D')\in\{w,ws\}$; moreover, if $\ell(ws)=\ell(w)+1$, then $\delta(C,D')=ws$.
\item \textbf{(Existence of descent)} If $\delta(C,D)=w$ and $s\in S$ with $\ell(ws)=\ell(w)-1$,
then there exists $D'\sim_s D$ such that $\delta(C,D')=ws$.
\end{enumerate}
\end{definition}

\begin{definition}[$(W,S)$-building]\label{def:building-abstract}
A \emph{$(W,S)$-building} is an $S$-chamber system $\mathcal{C}$ together with a Weyl distance
$\delta\colon\mathcal{C}\times\mathcal{C}\to W$.
An \emph{apartment} is a subset $A\subseteq\mathcal{C}$ equipped with a bijection
$\iota:A\to W$ such that $\iota$ is an isomorphism of $S$-chamber systems and
$\delta(C,D)=\delta_\Sigma(\iota(C),\iota(D))$ for all $C,D\in A$.
\end{definition}

\begin{remark}\label{rem:residues}
Given $J\subseteq S$ and a chamber $C\in\mathcal{C}$, the \emph{$J$-residue through $C$} is the set of all chambers
reachable from $C$ by a gallery whose type uses only letters from $J$.
\end{remark}

The Weyl distance controls minimal galleries exactly as in the Coxeter complex.

\begin{lemma}[Reduced words give minimal galleries]\label{lem:reduced-word-minimal-gallery}
Let $(\mathcal{C},\delta)$ be a $(W,S)$-building. Fix chambers $C,D\in\mathcal{C}$ and write $\delta(C,D)=w$.
\begin{enumerate}[label=(\alph*), leftmargin=8mm]
\item For any gallery $\mathbf{C}=(C=C_0,\ldots,C_n=D)$ of type $s_1\cdots s_n$,
the word $s_1\cdots s_n$ represents $w$ in $W$, and $n\ge \ell(w)$.
\item If $w=s_1\cdots s_r$ is a reduced expression, then there exists a gallery from $C$ to $D$
of type $s_1\cdots s_r$. In particular, the minimal gallery length from $C$ to $D$ equals $\ell(w)$.
\end{enumerate}
\end{lemma}

\begin{proof}
We sketch the standard argument based on (WD2)--(WD3).
For (a), start at $C_0=C$ and apply (WD2) iteratively along the steps
$C_{i-1}\sim_{s_i}C_i$ to see that $\delta(C,C_i)$ is obtained from $\delta(C,C_{i-1})$ by either staying
the same or multiplying by $s_i$ on the right; in particular, $\delta(C,D)$ lies in the subgroup word
generated by $s_1,\ldots,s_n$, and one checks that the product in $W$ of $s_1\cdots s_n$ must equal $w$.
Since any word for $w$ has length at least $\ell(w)$, we get $n\ge \ell(w)$.

For (b), write $w=s_1\cdots s_r$ reduced and argue by induction on $r$.
If $r=0$ there is nothing to prove. For $r\ge 1$, let $w'=ws_r$ so that $\ell(w')=r-1$.
By (WD3) applied to $s=s_r$ at the pair $(C,D)$, there exists $D'\sim_{s_r}D$ with $\delta(C,D')=w'$.
By induction there is a gallery from $C$ to $D'$ of type $s_1\cdots s_{r-1}$, and adjoining the final
$s_r$-step gives the required gallery to $D$.
\end{proof}

\begin{remark}\label{rem:distance-function}
Define the \emph{gallery distance} $d(C,D)$ as the minimal length of a gallery from $C$ to $D$.
Lemma~\ref{lem:reduced-word-minimal-gallery} shows that $d(C,D)=\ell(\delta(C,D))$.
This is the basic mechanism behind later ``length truncations'' and trace counts.
\end{remark}

\subsection{Buildings from affine BN-pairs (group-theoretic model)}\label{subsec:bn-pair-building}

We record the standard coset model that realizes many buildings as $G/I$ and encodes
the Weyl distance by Bruhat double cosets.

\begin{definition}[BN-pair and Weyl group]\label{def:bn-pair}
Let $G$ be a group. A \emph{BN-pair} is a pair of subgroups $(B,N)$ of $G$ such that
$G=\langle B,N\rangle$ and the quotient $W=N/(B\cap N)$ is generated by a set $S$ of involutions,
together with the usual axioms ensuring a Bruhat decomposition
\[
G=\bigsqcup_{w\in W} BwB
\]
and the standard multiplication rules with respect to simple reflections.
When $W$ is an affine Coxeter group, one speaks of an \emph{affine} BN-pair.
\end{definition}

In the affine BN-pair setting, $B$ plays the role of an Iwahori subgroup; we rename it $I$ to match later sections.

\begin{definition}[Iwahori and rank-$1$ parahorics]\label{def:iwahori-parahoric-strong}
Assume $G$ has an affine BN-pair $(I,N)$ with Weyl group $(W,S)$.
For each $s\in S$, fix a representative $\dot{s}\in N$ of $s$.
Define the \emph{rank-$1$ parahoric} subgroup
\[
P_s:=\langle I,\dot{s}\rangle.
\]
Equivalently, $P_s$ is the union $I\cup I\dot{s}I$ and satisfies $P_s/I$ finite in the locally finite case.
\end{definition}

\begin{definition}[Chambers and adjacency from parahorics]\label{def:coset-chambers}
Let $\mathcal{C}=G/I$.
For $s\in S$, define an $s$-adjacency relation on $\mathcal{C}$ by
\[
gI \sim_s g'I \quad \Longleftrightarrow \quad gP_s = g'P_s.
\]
\end{definition}

\begin{definition}[Weyl distance from double cosets]\label{def:delta-doublecoset}
Assume the Bruhat decomposition $G=\bigsqcup_{w\in W} IwI$.
Define $\delta:G/I\times G/I\to W$ by
\[
\delta(gI,hI)=w \quad\Longleftrightarrow\quad g^{-1}h\in IwI.
\]
\end{definition}

\begin{proposition}[Well-definedness and the one-step rule]\label{prop:delta-well-defined}
The map $\delta$ in Definition~\ref{def:delta-doublecoset} is well-defined.
Moreover, for $s\in S$ one has
\[
gI\sim_s hI \quad\Longleftrightarrow\quad \delta(gI,hI)\in\{1,s\}.
\]
\end{proposition}

\begin{proof}
Well-definedness follows from the disjointness of Bruhat double cosets $IwI$.
For the equivalence, note that $gI\sim_s hI$ means $g^{-1}h\in P_s$, and by
Definition~\ref{def:iwahori-parahoric-strong} we have $P_s\subseteq I\cup I\dot{s}I$,
so $g^{-1}h\in I$ or $I\dot{s}I$, i.e.\ $\delta(gI,hI)=1$ or $s$.
\end{proof}

\begin{thm}[Standard Bruhat--Tits fact]\label{thm:coset-building}
Under the axioms of an affine BN-pair (equivalently, the standard Bruhat--Tits setup),
the chamber system $G/I$ with adjacency as in Definition~\ref{def:coset-chambers}
and Weyl distance as in Definition~\ref{def:delta-doublecoset} is a thick, locally finite $(W,S)$-building.
\end{thm}

\begin{remark}\label{rem:why-this-matters}
This coset description is the natural home of ``galleries'' and ``Weyl distance'':
galleries become sequences of adjacent cosets, and $\delta$ is read off from Bruhat double cosets.
Later, our pointed-chamber dynamics refines this by introducing a directed state space.
\end{remark}

\section{The building of \texorpdfstring{$\operatorname{PGL}_3(F)$}{PGL3(F)}: simplicial realization and pointed chambers}\label{sec:pgl3-building-pointed}

In this section we specialize the abstract framework of Section~\ref{sec:chamber-system-building} to
$G=\operatorname{PGL}_3(F)$.
We record (i) the group-theoretic chamber model $G/I$ of type $\widetilde{A}_2$,
(ii) the standard simplicial realization via lattice classes in $F^3$,
and (iii) the pointed chamber language (types, directed adjacency, tailless pointed galleries)
used in the rest of the paper.

\subsection{Setup and Weyl group type}\label{subsec:pgl3-setup}

Let $F$ be a non-Archimedean local field with discrete valuation $\nu$,
let $\mathcal{O}$ be its valuation ring, and fix a uniformizer $\pi\in\mathcal{O}$.
Write $k=\mathcal{O}/\pi\mathcal{O}$ and $q=|k|$.
Let
\[
G=\operatorname{PGL}_3(F)=\mathrm{GL}_3(F)/F^\times,
\]
and let $\mathcal{B}$ denote the associated Bruhat--Tits building.
It is an affine building of type $\widetilde{A}_2$.
We write $(W,S)$ for the affine Coxeter system of type $\widetilde{A}_2$,
with $S=\{s_0,s_1,s_2\}$.

\subsection{The chamber system model \texorpdfstring{$G/I$}{G/I}}\label{subsec:GI-model}

Fix an Iwahori subgroup $I\le G$ (equivalently, the stabilizer of a chosen alcove in $\mathcal{B}$).
For each $s_i\in S$, let $P_i$ denote the corresponding maximal parahoric subgroup containing $I$.
Then, as in Section~\ref{subsec:bn-pair-building}:

\begin{itemize}
\item the set of chambers (alcoves) is $[\mathbf{c}]=G/I$;
\item two chambers $gI$ and $g'I$ are $s_i$-adjacent if and only if $gP_i=g'P_i$;
\item the Weyl distance is determined by Bruhat double cosets:
\[
\delta(gI,hI)=w \quad \Longleftrightarrow \quad g^{-1}h\in IwI.
\]
\end{itemize}

Thus, by Theorem~\ref{thm:coset-building}, the chamber system $G/I$ is a thick, locally finite $(W,S)$-building.

\begin{remark}\label{rem:panels-as-parahoric-cosets}
The $s_i$-panels through a fixed chamber $gI$ are parametrized by the cosets in $gP_i/I$.
In particular, thickness and local finiteness are controlled by the finite quotients $P_i/I$.
\end{remark}

\subsection{Simplicial realization via lattice classes in \texorpdfstring{$F^3$}{F3}}\label{subsec:simplicial-lattice}

We now recall the standard simplicial model, which is convenient for explicit coordinates and truncations.

\begin{definition}[Vertices and chambers]\label{def:bt-lattice-model}
Two rank-$3$ $\mathcal{O}$-lattices $L,L'\subset F^3$ are \emph{homothetic} if $L=aL'$ for some $a\in F^\times$.
Vertices of $\mathcal{B}$ are homothety classes $[L]$.

Three vertices $[L_1],[L_2],[L_3]$ form a $2$-simplex (an alcove/chamber) if there exist representatives
$L_i'\in [L_i]$ such that
\[
\pi L_1'\subset L_3'\subset L_2'\subset L_1'.
\]
\end{definition}

The action of $G$ on $F^3$ induces a simplicial action on $\mathcal{B}$ by $g[L]=[gL]$.
Moreover, $\mathcal{B}$ is (canonically) the simplicial realization of the chamber system $G/I$:
alcoves in $G/I$ correspond to $2$-simplices in $\mathcal{B}$, and parahoric cosets correspond to panels.

\begin{definition}[Vertex type map]\label{def:vertex-type-map}
Fix the standard lattice $L_0=\mathcal{O}^3\subset F^3$.
For a vertex $[L]\in V(\mathcal{B})$, choose $g\in \mathrm{GL}_3(F)$ such that $L=gL_0$ and define
\[
\tau([L]) \equiv \nu(\det g)\pmod{3}\in \mathbb{Z}/3\mathbb{Z}.
\]
\end{definition}

\begin{lemma}\label{lem:type-map-well-defined}
The map $\tau:V(\mathcal{B})\to \mathbb{Z}/3\mathbb{Z}$ in Definition~\ref{def:vertex-type-map} is well-defined.
Moreover, for any alcove (2-simplex) in $\mathcal{B}$, its three vertices have pairwise distinct types.
\end{lemma}

\begin{proof}
If $L=gL_0=g'L_0$, then $g^{-1}g'\in \mathrm{GL}_3(\mathcal{O})$, hence $\nu(\det(g^{-1}g'))=0$ and
$\nu(\det g)\equiv \nu(\det g')\pmod{3}$.
If we replace $L$ by a homothetic lattice $aL$ with $a\in F^\times$, then one may replace $g$ by $ag$,
so $\nu(\det(ag))=\nu(a^3\det g)\equiv \nu(\det g)\pmod{3}$, proving well-definedness on $[L]$.

For the second claim, a chamber corresponds to a chain
$\pi L_1'\subset L_3'\subset L_2'\subset L_1'$.
After choosing $g$ with $L_1'=gL_0$, the intermediate lattices correspond to standard index-$q$ and index-$q^2$
sublattices between $\pi L_0$ and $L_0$, so their determinants differ by valuation shifts $0,1,2$ modulo $3$.
Thus the three vertex types are distinct in $\mathbb{Z}/3\mathbb{Z}$.
\end{proof}

\begin{remark}\label{rem:vertex-types}
Each alcove in $\mathcal{B}$ has three vertices of distinct types in $\mathbb{Z}/3\mathbb{Z}$.
Consequently, an ordering of the vertices determines an orientation (clockwise or counterclockwise) of type increments.
This will be encoded using pointed chambers below.
\end{remark}

\subsection{Pointed chambers and directed pointed galleries}\label{subsec:pointed-chambers-galleries}

\begin{definition}[Pointed chamber]\label{def:pointed-chamber-pgl3}
A \emph{pointed chamber} is an ordered triple of vertices
\[
c=(v_1,v_2,v_3)
\]
such that $\{v_1,v_2,v_3\}$ is the vertex set of a chamber (2-simplex) in $\mathcal{B}$.
The order is part of the data.
\end{definition}

\begin{definition}[Type of a pointed chamber]\label{def:type-pointed-chamber-pgl3}
Let $c=(v_1,v_2,v_3)$ be a pointed chamber.
For $k\in\{1,2\}$, we say that $c$ has \emph{type $k$} if there exists $i\in\mathbb{Z}/3\mathbb{Z}$ such that
\[
(\tau(v_1),\tau(v_2),\tau(v_3))\equiv (i,\ i+k,\ i+2k)\pmod{3}.
\]
Equivalently, $c$ has type $k$ if
\[
\tau(v_2)-\tau(v_1)\equiv k
\quad\text{and}\quad
\tau(v_3)-\tau(v_2)\equiv k
\pmod{3}.
\]
\end{definition}

\begin{definition}[Type $k$ edge-adjacency]\label{def:type-k-edge-adjacency-pgl3}
Let $c=(v_1,v_2,v_3)$ and $c'=(v_1',v_2',v_3')$ be pointed chambers of the same type $k$.
We say that $c'$ is \emph{type $k$ edge-adjacent} to $c$ if
\[
v_1'=v_2,\qquad v_2'=v_3.
\]
Equivalently, $c'$ is obtained from $c$ by crossing the (unpointed) edge $\{v_2,v_3\}$ and keeping the same orientation.
\end{definition}

\begin{definition}[Type $k$ pointed chamber gallery]\label{def:type-k-pointed-gallery-pgl3}
A \emph{type $k$ pointed chamber gallery} of length $n$ is a sequence
\[
\mathbf{c}=(c_0,\ldots,c_{n-1},c_n),\qquad c_i=(v_{1,i},v_{2,i},v_{3,i}),
\]
such that $c_{i+1}$ is type $k$ edge-adjacent to $c_i$ for every $i$.
Equivalently,
\[
v_{2,i}=v_{1,i+1},\qquad v_{3,i}=v_{2,i+1}\quad\text{for all }i.
\]
\end{definition}

\begin{definition}[Tailless and closed]\label{def:tailless-closed}
A type $k$ pointed chamber gallery $\mathbf{c}=(c_0,\ldots,c_{n-1},c_n)$ is \emph{tailless} if
\[
v_{1,i}\neq v_{3,i+1}\quad\text{for all }i,
\]
and it is \emph{closed} if $c_n=c_0$ (equivalently, the last step $c_{n-1}\to c_n$ closes the gallery), i.e.,
\[
v_{1,0}=v_{2,n-1},\qquad v_{2,0}=v_{3,n-1}.
\]
\end{definition}

\begin{definition}[Shift equivalence]\label{def:shift-equivalence-pgl3}
For a closed gallery $\mathbf{c}=(c_0,\ldots,c_{n-1},c_n)$, define
\[
\sigma(\mathbf{c})=(c_1,\ldots,c_n,c_0).
\]
Two closed galleries are \emph{equivalent} if they differ by a cyclic shift:
$\mathbf{c}_1=\sigma^j(\mathbf{c}_2)$ for some integer $j$.
\end{definition}

\begin{remark}[Forgetting the order gives an ordinary gallery]\label{rem:pointed-refines-standard-gallery}
For a pointed chamber $c=(v_1,v_2,v_3)$, write $|c|=\{v_1,v_2,v_3\}$ for the underlying (unpointed) chamber.
If $\mathbf{c}=(c_0,\ldots,c_{n-1},c_n)$ is a type $k$ pointed chamber gallery, then
\[
(|c_0|,\ldots,|c_{n-1}|,|c_n|)
\]
is an (unpointed) chamber gallery in the sense of Definition~\ref{def:galleries}.
Thus pointed galleries refine the chamber-system language by introducing a directed state space.
\end{remark}

\begin{lemma}[The induced Coxeter type is cyclic]\label{lem:pointed-implies-coxeter-type-cyclic}
Let $\mathbf{c}=(c_0,\ldots,c_{n-1},c_n)$ be a type $1$ pointed chamber gallery and write $c_i=(v_{1,i},v_{2,i},v_{3,i})$.
Set $i_0=\tau(v_{1,1})\in \mathbb{Z}/3\mathbb{Z}$.
Then the underlying unpointed gallery $(|c_0|,\ldots,|c_{n-1}|,|c_n|)$ is of Coxeter type
\[
s_{i_0}\, s_{i_0+1}\, s_{i_0+2}\, s_{i_0}\, s_{i_0+1}\,\cdots
\]
(truncated to length $n$), where indices are taken modulo $3$.
For a type $2$ pointed chamber gallery, the induced Coxeter type is the reverse cyclic word.
\end{lemma}

\begin{proof}
For each $i$, the step $c_i\to c_{i+1}$ crosses the edge $\{v_{2,i},v_{3,i}\}$ and replaces the vertex $v_{1,i}$.
In an $\widetilde{A}_2$ building, the panel type is determined by the type of the omitted vertex.
Since $c_i$ has type $1$, we have
\[
\tau(v_{1,i}),\tau(v_{2,i}),\tau(v_{3,i}) \equiv (i_0+i-1,\ i_0+i,\ i_0+i+1)\pmod{3}.
\]
Hence the crossed panel at step $i$ has type $\tau(v_{1,i})\equiv i_0+i-1$, which corresponds to the generator
$s_{i_0+i-1}$ in our indexing. This yields the stated cyclic word.
The type $2$ case is analogous and reverses the cyclic orientation.
\end{proof}

\subsection{Thickness and local extension of type $1$ tailless pointed galleries}\label{subsec:local-extension-pgl3}

The link of any vertex of $\mathcal{B}$ is the spherical building of $\operatorname{PGL}_3(k)$,
namely the projective plane $\mathrm{PG}(2,q)$.
A panel corresponds to a partial flag, and it extends to a complete flag in exactly $q+1$ ways.
Hence each panel is contained in $q+1$ chambers, and $\mathcal{B}$ has thickness $q+1$.

\begin{lemma}\label{lem:local-extension-pgl3}
Let $(\widetilde{c}_0,\ldots,\widetilde{c}_{n-1},\widetilde{c}_n)$ be a type $1$ tailless pointed chamber gallery in $\mathcal{B}$.
Then there exist exactly $q$ distinct pointed chambers $\widetilde{c}_{n+1}$ such that
$(\widetilde{c}_0,\ldots,\widetilde{c}_n,\widetilde{c}_{n+1})$ is again a type $1$ tailless pointed chamber gallery.
\end{lemma}

\begin{proof}
Write $\widetilde{c}_n=(v_1,v_2,v_3)$.
To extend $\widetilde{c}_n$ by one type $1$ step, we must choose a vertex $v_4$ such that
$\{v_2,v_3,v_4\}$ is a chamber and the pointed chamber
\[
\widetilde{c}_{n+1}=(v_2,v_3,v_4)
\]
has type $1$. The type condition forces $\tau(v_4)=\tau(v_1)$ (since types advance cyclically by $+1$).
Now the edge $\{v_2,v_3\}$ is a panel, and there are exactly $q+1$ chambers containing it.
Among these $q+1$ choices, exactly one chamber is the ``backtracking'' chamber with third vertex $v_4=v_1$,
which yields the pointed chamber $(v_2,v_3,v_1)$ and violates the tailless condition
$v_1\neq v_{3,n+1}$ in Definition~\ref{def:tailless-closed}.
All remaining $q$ chambers give distinct vertices $v_4\neq v_1$, hence distinct pointed chambers
$\widetilde{c}_{n+1}=(v_2,v_3,v_4)$, and they are automatically tailless by construction.
\end{proof}


\section{Admissible galleries in the non-uniform quotient}\label{sec:3}
In this section, we describe the fundamental domain of $\Gamma$ and the image of tailless galleries in the quotient space by the standard non-uniform lattice $\Gamma$ of $G$.

\subsection{Fundamental domain} Let $\mathbb{F}_q$ be a finite field of order $q$. Let $\mathbb{F}_q[t]$ and $\mathbb{F}_q(t)$ be the ring of polynomials and the field of rational functions, respectively. The absolute value of the field $\mathbb{F}_q(t)$ is defined for any rational function $f=\frac{g}{h}$ by
$$\|f\|:=q^{\deg{g}-\deg{h}}.$$
The \textit{field of formal Laurent series} $\mathbb{F}_q(\!(t^{-1})\!)$ in $t^{-1}$ is the completion of $\mathbb{F}_q(t)$ with respect to $\|\cdot\|$, i.e.
$$\mathbb{F}_q(\!(t^{-1})\!):=\left\{\sum_{n=-N}^\infty a_nt^{-n}:N\in \mathbb{Z},a_n\in \mathbb{F}_q\right\}.$$
The valuation ring $\mathcal{O}$ is the subring of power series $\mathbb{F}_q[\![t^{-1}]\!]$
$$\mathbb{F}_q[\![t^{-1}]\!]:=\left\{\sum_{n=0}^\infty a_nt^{-n}:a_n\in \mathbb{F}_q\right\}.$$
Let $G:=\PGL_3(\mathbb{F}_q(\!(t^{-1})\!))$ and $K:=\PGL_3(\mathbb{F}_q[\![t^{-1}]\!])$. Denote by $\Gamma:=\operatorname{PGL}_3(,\mathbb{F}_q[t])$ the image of the canonical projection map $\operatorname{GL}_3(,\mathbb{F}_q[t])\rightarrow G.$

We recall that (e.g. Lemma 3.2 in \cite{HK1}) for any $g\in G$, there exists a unique pair of non-negative integers $(m,n)$ with $m\geq n$ such that 
$$g\in \Gamma \diag{t^m,t^n,1}K.$$
 Let $\mathsf{x}_{m,n}$ be the vertex corresponding to $\Gamma \diag{t^m,t^n,1}K$.
\begin{figure}[H]
\begin{center}
\begin{tikzpicture}[scale=0.52]
  \draw [thick](-5,0) -- (5.5,0);   \draw [thick](-0.5,+{sqrt(60.75)}) -- (-5,0);\draw[thick](-4,+{sqrt(3)}) -- (-3,0);\draw[thick] (-3,+{sqrt(12)}) -- (-1,0);\draw [thick](-2,{sqrt(27)}) -- (1,0);\draw[thick] (-1,{sqrt(48)})--(3,0);\draw[thick] (-4,+{sqrt(3})--(5.5,+{sqrt(3)});\draw[thick](-3,+{sqrt(12)})--(5.5,+{sqrt(12)});\draw[thick](-2,+{sqrt(27)})--(5.5,+{sqrt(27)});\draw[thick](-1,+{sqrt(48)})--(-1,+{sqrt(48)});\draw[thick](1.5,+{sqrt(60.75)})--(-3,0);\draw[thick](3.5,+{sqrt(60.75)})--(-1,0);\draw[thick](5.5,+{sqrt(60.75)})--(1,0);\draw[thick](5.5,+{sqrt(75)}/2)--(3,0);\draw[thick](5,0)--(0.5,+{sqrt(60.75)});\draw[thick] (-1,+{sqrt(48)})--(5.5,+{sqrt(48)});\draw[thick](5.5,+{sqrt(27)}/2)--(2.5,+{sqrt(60.75)});\draw[thick] (4.5,+{sqrt(60.75)})--(5.5,+{sqrt(147)}/2);\draw[thick] (5,0)--(5.5,+{sqrt(3)}/2);
\node at (-5,-0.3) {$\mathsf{x}_{0,0}$};\node at (-3,-0.3) {$\mathsf{x}_{1,0}$};\node at (-1,-0.3) {$\mathsf{x}_{2,0}$};\node at (1,-0.3) {$\mathsf{x}_{3,0}$};\node at(3,-0.3) {$\mathsf{x}_{4,0}$};\node at (5,-0.3) {$\mathsf{x}_{5,0}$};\node at (6.5,3) {$\cdots$};\node at (-4.6,1.9) {$\mathsf{x}_{1,1}$};\node at (-2.6,1.9) {$\mathsf{x}_{2,1}$};\node at (-0.6,1.9) {$\mathsf{x}_{3,1}$};\node at (1.4,1.9) {$\mathsf{x}_{4,1}$};\node at (3.4,1.9) {$\mathsf{x}_{5,1}$};\node at (-3.6,3.65) {$\mathsf{x}_{2,2}$};\node at (-1.6,3.65) {$\mathsf{x}_{3,2}$};\node at (0.4,3.65) {$\mathsf{x}_{4,2}$};\node at(2.4,3.65) {$\mathsf{x}_{5,2}$};\node at (4.4,3.65) {$\mathsf{x}_{6,2}$};\node at (-2.6,5.35) {$\mathsf{x}_{3,3}$};\node at (-0.6,5.35) {$\mathsf{x}_{4,3}$};\node at (1.4,5.35) {$\mathsf{x}_{5,3}$};\node at (3.4, 5.35) {$\mathsf{x}_{6,3}$};\node at (-1.6, 7.1) {$\mathsf{x}_{4,4}$};\node at (0.4, 7.1) {$\mathsf{x}_{5,4}$};\node at (2.4, 7.1) {$\mathsf{x}_{6,4}$};\node at (4.4, 7.1) {$\mathsf{x}_{7,4}$};
\fill (-5,0)    circle (3pt); \fill (-3,0)    circle (3pt); \fill (-1,-0)    circle (3pt);\fill (1,-0)    circle (3pt); \fill (3,0) circle (3pt);\fill(5,0) circle (3pt);\fill (-4,+{sqrt(3)})    circle (3pt); \fill (-3,+{sqrt(12)})    circle (3pt); \fill (-2,{sqrt(27)})    circle (3pt);\fill (-1,+{sqrt(48)})    circle (3pt); \fill (-2,+{sqrt(3)}) circle (3pt);\fill(-0,+{sqrt(3)}) circle (3pt);\fill (0,+{sqrt(3)}) circle (3pt);\fill(2,{sqrt(3)}) circle (3pt);\fill (4,+{sqrt(3)}) circle (3pt);\fill (-3,{sqrt(12)}) circle (3pt); \fill (-1,{sqrt(12)}) circle (3pt);\fill (1,{sqrt(12)}) circle (3pt);\fill(3,{sqrt(12)}) circle (3pt);\fill (3,{sqrt(12)}) circle (3pt);\fill(5,{sqrt(12)}) circle (3pt);\fill (-2,{sqrt(27)}) circle (3pt);\fill (0,{sqrt(27)}) circle (3pt);\fill (2,{sqrt(27)}) circle (3pt); \fill (4,{sqrt(27)}) circle (3pt);\fill(1,{sqrt(48)}) circle (3pt); \fill(3,{sqrt(48)}) circle (3pt);\fill (5,{sqrt(48)}) circle (3pt);
\end{tikzpicture}
\end{center}
\caption{The fundamental domain for $\Gamma\backslash\mathcal{B}$}\label{fig:funddom}
\end{figure}
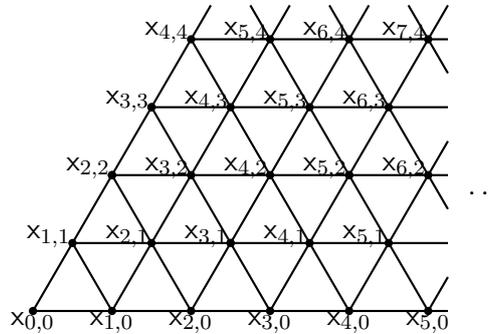

Each vertex $\mathsf{x}_{m,n}$ is adjacent to $\mathsf{x}_{m',n'}$ if and only if the following holds:
\begin{displaymath}
\begin{cases} (m',n')\in\{(m\pm1,n),(m,n\pm1),(m\pm1,n\pm1)\} &\text{ if }m>n>0\\
(m',n')\in \{(m\pm1,n),(m,n+1),(m+1,n+1)\} &\text{ if }m>n=0\\
(m',n')\in \{(m+1,n),(m,n-1),(m\pm1,n\pm1)\} &\text{ if }m=n>0\\
(m',n')\in \{(1,0),(1,1)\}&\text{ if }m=n=0.
\end{cases}
\end{displaymath}

\subsection{Admissible galleries in the quotient $\Gamma \backslash \mathcal{B}$}\label{sec:3.2} In this subsection, we define the weight of pointed chambers to count the number of closed galleries arising as quotients of tailless galleries in $\mathcal{B}$. Since the action of $\Gamma$ preserves the color difference between the endpoints of each directed edge, the definition of type of pointed chamber and galleries in the quotient space are well defined, respectively.
 A gallery $\mathbf{c}$ in $\Gamma\backslash \mathcal{B}$ is called \emph{admissible} if it is the projection of tailless gallery in $\mathcal{B}.$
 
We define the weight of an adjacency step in a gallery $\mathbf{c}$. For a pointed chamber $\widetilde{c}$ of $\mathcal{B}$ and a pointed chamber $c'$ of $X$, let
$$w(\widetilde{c},c')=\#\{\widetilde{c}'\colon\widetilde{c}\textrm{ is one of the lifts }\textrm {of }c',(\widetilde{c},\widetilde{c}') \textrm{ is a tailless gallery in }\mathcal{B}\}.$$ 
Let $\widetilde{c}$ be a lift of $c$ in $\mathcal{B}$. The weight of an adjacency step $(c,c')$ is defined by $$w(c,c')=w(\widetilde{c},c').$$ 


\begin{lem}[Well-definedness of the weight]\label{lem:weight_welldefined}
Let $X=\Gamma\backslash\mathcal{B}$ and let $c,c'$ be edge-adjacent type~$1$ pointed chambers in $X$.
Fix a lift $\tilde c$ of $c$ to $\mathcal{B}$, and define $w_{\tilde c}(c,c')$ to be the number of
type~$1$ pointed chambers $\tilde c'$ in $\mathcal{B}$ such that:
\begin{enumerate}
\item $\tilde c'$ is edge-adjacent to $\tilde c$ in $\mathcal{B}$,
\item the projection of $\tilde c'$ to $X$ equals $c'$,
\item the step $\tilde c\to \tilde c'$ is tailless (in the sense of Definition~\ref{def:tailless-closed}).
\end{enumerate}
Then $w_{\tilde c}(c,c')$ is independent of the choice of the lift $\tilde c$.
In particular, the weight
\[
w(c,c'):=w_{\tilde c}(c,c')
\]
is well-defined.
\end{lem}

\begin{proof}
Let $\tilde c_1$ and $\tilde c_2$ be two lifts of $c$.
Then there exists $\gamma\in\Gamma$ such that $\gamma\tilde c_1=\tilde c_2$.
The map $\tilde c'\mapsto \gamma\tilde c'$ defines a bijection between the set of pointed chambers
$\tilde c'$ satisfying (1)--(3) for $\tilde c_1$ and the corresponding set for $\tilde c_2$:
\begin{itemize}
\item $\Gamma$ acts by simplicial automorphisms, hence preserves edge-adjacency;
\item the projection $\mathcal{B}\to \Gamma\backslash\mathcal{B}$ is $\Gamma$-invariant, so $\tilde c'$ projects to $c'$
iff $\gamma\tilde c'$ projects to $c'$;
\item taillessness is a local combinatorial condition and is preserved by automorphisms.
\end{itemize}
Therefore the two sets have the same cardinality, and $w_{\tilde c}(c,c')$ does not depend on the lift.
\end{proof}


Let $\Delta_{m,n}$ be the (unpointed) chamber with vertices
\[
\mathsf{x}_{m,n},\; \mathsf{x}_{m+1,n},\; \mathsf{x}_{m+1,n+1}.
\]
A \emph{type $1$ pointed chamber} is an \emph{ordered} triple of the vertices of $\Delta_{m,n}$ whose first vertex has type $1$.
Accordingly, $\Delta_{m,n}$ gives rise to three type $1$ pointed chambers, which we denote by
\begin{align*}
c_{m,n,1}=&\,(\mathsf{x}_{m,n},\,\mathsf{x}_{m+1,n},\,\mathsf{x}_{m+1,n+1}),\\
c_{m,n,2}=&\,(\mathsf{x}_{m+1,n},\,\mathsf{x}_{m+1,n+1},\,\mathsf{x}_{m,n}),\\
c_{m,n,3}=&\,(\mathsf{x}_{m+1,n+1},\,\mathsf{x}_{m,n},\,\mathsf{x}_{m+1,n}).
\end{align*}

Similarly, let $\nabla_{m,n}$ be the (unpointed) chamber with vertices
\[
\mathsf{x}_{m+1,n},\; \mathsf{x}_{m+1,n+1},\; \mathsf{x}_{m+2,n+1}.
\]
It also determines three type $1$ pointed chambers, denoted by
\begin{align*}
d_{m,n,1}=&\,(\mathsf{x}_{m+1,n},\,\mathsf{x}_{m+1,n+1},\,\mathsf{x}_{m+2,n+1}),\\
d_{m,n,2}=&\,(\mathsf{x}_{m+1,n+1},\,\mathsf{x}_{m+2,n+1},\,\mathsf{x}_{m+1,n}),\\
d_{m,n,3}=&\,(\mathsf{x}_{m+2,n+1},\,\mathsf{x}_{m+1,n},\,\mathsf{x}_{m+1,n+1}).
\end{align*}

For each type $1$ pointed chamber $c_{m,n,i}$ (resp.\ $d_{m,n,i}$), we list all type $1$ pointed chambers in $\mathcal{B}$
that share an edge with it (see Figures~\ref{fig4} and~\ref{fig:5}).
Under the quotient map to $\Gamma\backslash\mathcal{B}$, each such neighbor is represented by some
$c_{m',n',i'}$ or $d_{m'',n'',i''}$.
The weights of these adjacencies are as follows:
\[
\begin{array}{lll}
w(c_{m,n,1},c_{m,n,2})=q-1, &\qquad& w(c_{m,n,1},d_{m,n,1})=1,\\[2pt]
w(c_{m,n,2},c_{m,n,3})=q \ \text{if } m=n, && w(c_{m,n,2},d_{m-1,n,3})=q \ \text{if } m\neq n,\\[2pt]
w(c_{m,n,3},c_{m,n,1})=q \ \text{if } n=0, && w(c_{m,n,3},d_{m-1,n-1,2})=q \ \text{if } n\neq 0,
\end{array}
\]
and
\[
\begin{array}{lll}
w(d_{m,n,1},c_{m+1,n+1,1})=1, &\qquad& w(d_{m,n,1},d_{m,n,2})=q-1,\\[2pt]
w(d_{m,n,2},c_{m+1,n,3})=1, && w(d_{m,n,2},d_{m,n,3})=q-1,\\[2pt]
w(d_{m,n,3},c_{m,n,2})=q, && w(d_{m,n,3},d_{m,n,1})=0.
\end{array}
\]

\medskip
\noindent
Here the value $w(\cdot,\cdot)=0$ means that, although the two pointed chambers are edge-adjacent in the quotient,
there is \emph{no} lift of this adjacency to a \emph{tailless} (type $1$) step in the building $\mathcal{B}$.
Equivalently, the corresponding transition is forbidden by the taillessness condition, and hence it contributes
no admissible gallery and is assigned weight zero. In particular, the adjacency $(d_{m,n,3},d_{m,n,1})$ is of this
forbidden type, so we set $w(d_{m,n,3},d_{m,n,1})=0$.

\begin{center}
\begin{figure}[H]
\label{fig4}
\newcommand*\rows{6}
\begin{tikzpicture}[scale=2.1]
\draw[thick] (-2,0)--(0,0);
\filldraw[pattern=north east lines, pattern color=gray] (0,0)--(-1,+{sqrt(3)})--(-2,0);
\draw[thick]  (0,0)--(1,+{sqrt(3)})--(-1,+{sqrt(3)});
\draw[thick] (0,0)--(-0,+{sqrt(12)}/3)--(-1,+{sqrt(3)});
\draw[thick] (0,0)--(-1,-{sqrt(3)}/3)--(-2,0);
\draw[thick] (-2,0)--(-2,+{sqrt(12)}/3)--(-1,+{sqrt(3)});
\draw[thick] (-2,0)--(-3,{sqrt(3)})--(-1,{sqrt(3)});
\draw[thick] (-2,0)--(-1,-{sqrt(3)})--(0,0);
\node at(-2.7,0) {\Tiny{$\begin{pmatrix}t^m&0&0\\0&t^n&0\\0&0&1\end{pmatrix}K$}};
\node at(0.8,-0) {\Tiny{$\begin{pmatrix}t^{m+1}&0&0\\0&t^n&0\\0&0&1\end{pmatrix}K$}};
\node at(-1,2.1) {\Tiny{$\begin{pmatrix}t^{m+1}&0&0\\0&t^{n+1}&0\\0&0&1\end{pmatrix}K$}};
\node at(-2.8,1.2) {\Tiny{$\begin{pmatrix}t^{m}&t^{m+1}&0\\0&t^{n+1}&0\\0&0&1\end{pmatrix}K$}};
\node at(0.8,1.2) {\Tiny{$\begin{pmatrix}t^m&0&t^{m+1}\\0&t^n&0\\0&0&1\end{pmatrix}K$}};
\node at(1.2,2.1) {\Tiny{$\begin{pmatrix}t^{m+2}&0&0\\0&t^{n+1}&0\\0&0&1\end{pmatrix}K$}};
\node at(-3.2,2.1) {\Tiny{$\begin{pmatrix}t^{m}&0&0\\0&t^{n+1}&0\\0&0&1\end{pmatrix}K$}};
\node at(-1,-0.9) {\Tiny{$\begin{pmatrix}t^{m}&0&0\\0&t^{n-1}&t^n\\0&0&1\end{pmatrix}K$}};
\node at(-1,-2) {\Tiny{$\begin{pmatrix}t^{m}&0&0\\0&t^{n-1}&0\\0&0&1\end{pmatrix}K$}};
\end{tikzpicture}
\caption{The underlying chamber $\Delta_{m,n}$ of $c_{m,n,i}$}\label{fig:3}
\end{figure}
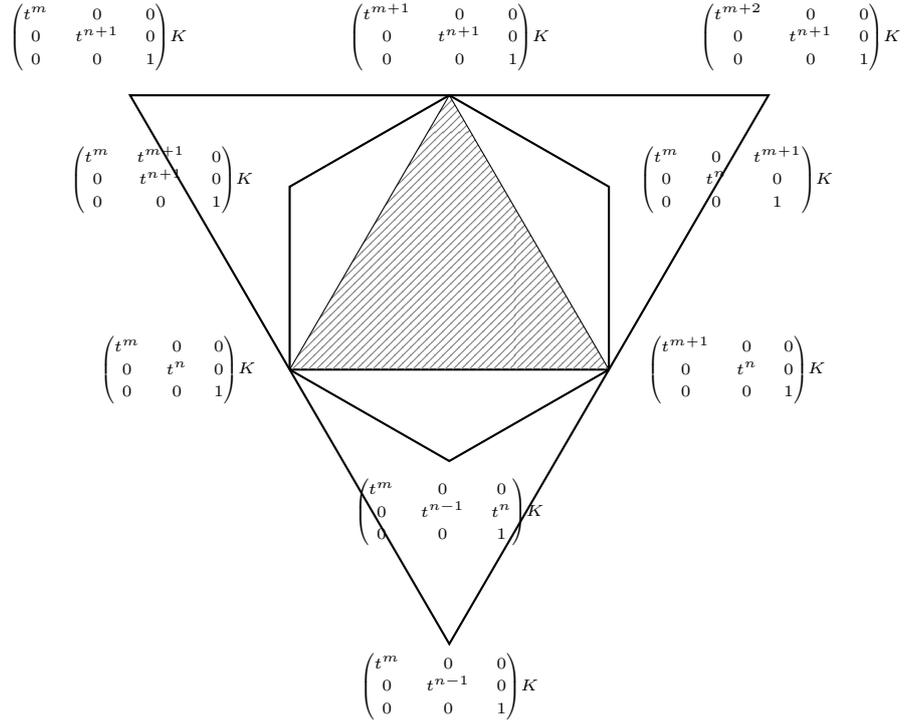
\end{center}
\begin{center}
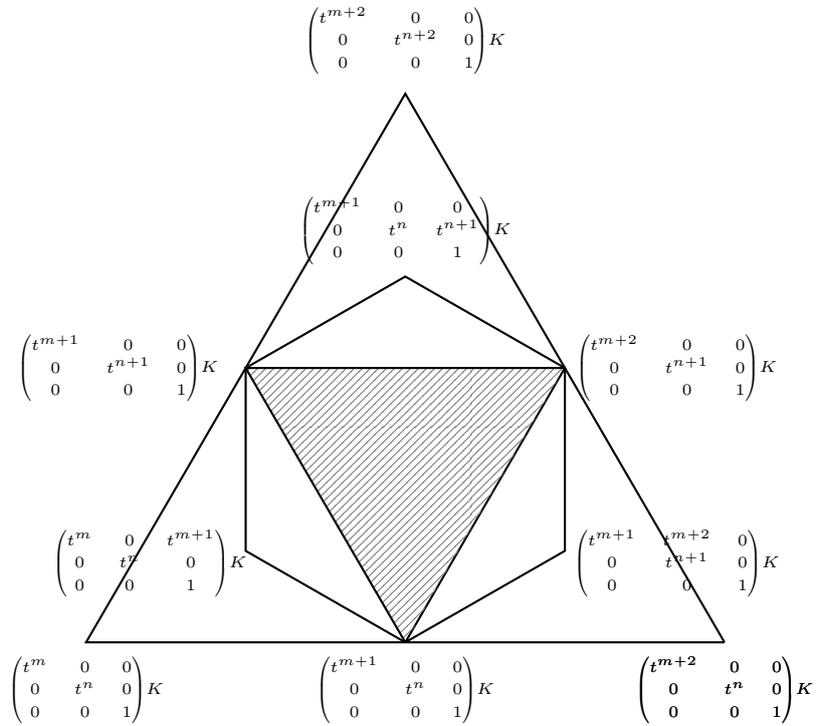
\begin{figure}[H]
\label{fig:5}
\newcommand*\rows{6}
\begin{tikzpicture}[scale=2.1]
\draw[thick] (-2,0)--(0,0);
\draw[thick] (0,0)--(-1,+{sqrt(3)})--(-2,0);
\draw[thick] (-2,0)--(-1,+{sqrt(3)}/3)--(-1,+{sqrt(3)});
\draw[thick] (-2,0)--(-3,+{sqrt(3)}/3)--(-3,+{sqrt(3)});
\filldraw[pattern=north east lines, pattern color=gray] (-2,0)--(-3,{sqrt(3)})--(-1,{sqrt(3)});
\draw[thick] (-2,0)--(-3,{sqrt(3)})--(-1,{sqrt(3)});
\draw[thick] (-2,0)--(-4,0)--(-3,+{sqrt(3)});
\draw[thick] (-3,{sqrt(3)})--(-2,+{sqrt(12)})--(-1,+{sqrt(3)});
\draw[thick] (-3,{sqrt(3)})--(-2,+{sqrt(48)}/3)--(-1,+{sqrt(3)});
\node at(-2,-0.3) {\Tiny{$\begin{pmatrix}t^{m+1}&0&0\\0&t^{n}&0\\0&0&1\end{pmatrix}K$}};
\node at(-4,-0.3) {\Tiny{$\begin{pmatrix}t^{m}&0&0\\0&t^{n}&0\\0&0&1\end{pmatrix}K$}};
\node at(0,-0.3) {\Tiny{$\begin{pmatrix}t^{m+2}&0&0\\0&t^{n}&0\\0&0&1\end{pmatrix}K$}};
\node at(0,-0.3) {\Tiny{$\begin{pmatrix}t^{m+2}&0&0\\0&t^{n}&0\\0&0&1\end{pmatrix}K$}};
\node at(-0.3,{sqrt(3)}) {\Tiny{$\begin{pmatrix}t^{m+2}&0&0\\0&t^{n+1}&0\\0&0&1\end{pmatrix}K$}};
\node at(-3.8,{sqrt(3)}) {\Tiny{$\begin{pmatrix}t^{m+1}&0&0\\0&t^{n+1}&0\\0&0&1\end{pmatrix}K$}};
\node at(-2,3.8) {\Tiny{$\begin{pmatrix}t^{m+2}&0&0\\0&t^{n+2}&0\\0&0&1\end{pmatrix}K$}};
\node at(-0.3,0.5) {\Tiny{$\begin{pmatrix}t^{m+1}&t^{m+2}&0\\0&t^{n+1}&0\\0&0&1\end{pmatrix}K$}};
\node at(-3.6,0.5) {\Tiny{$\begin{pmatrix}t^{m}&0&t^{m+1}\\0&t^{n}&0\\0&0&1\end{pmatrix}K$}};
\node at(-2,2.6) {\Tiny{$\begin{pmatrix}t^{m+1}&0&0\\0&t^{n}&t^{n+1}\\0&0&1\end{pmatrix}K$}};
\end{tikzpicture}
\caption{The underlying chamber $\nabla_{m,n}$ of $d_{m,n,i}$}\label{fig:4}
\end{figure}
\end{center}

\medskip
\noindent

The following lemma plays a key role in the proof of the determinant formula, as it
provides the required local finiteness of closed admissible gallery classes.

\begin{lem}[Finiteness of admissible closed gallery classes]\label{lem:3.1}For any positive integer $N$, there exist finitely many type 1 admissible closed gallery classes of length $N$ in the quotient $\Gamma\backslash\mathcal{B}$.
\end{lem}

\begin{proof}
Fix $N\ge 1$.  Let
\[
\mathbf{c}=(x_0,\dots,x_{N-1},x_N)
\]
be a type $1$ admissible closed gallery of length $N$ in $X=\Gamma\backslash\mathcal B$, so that each step
$x_j\to x_{j+1}$ (indices mod $N$) is one of the adjacency transitions listed in the weight table of
Subsection~4.2 (equivalently, $w(x_j,x_{j+1})\neq 0$).

\medskip
\noindent\textbf{Step 1: a uniform bound on the second coordinate $n$.}
We claim that every chamber occurring in $\mathbf{c}$ has second coordinate $n\le N$.

Indeed, suppose first that $\mathbf{c}$ contains a pointed chamber of the form $c_{m,n,1}$ or $d_{m,n,1}$ with $n\ge 1$.
From the weight table we have
\[
w(d_{m-1,n-1,1},c_{m,n,1})=1,\qquad w(c_{m,n,1},d_{m,n,1})=1,
\]
and, moreover, these are the only admissible incoming transitions into $c_{m,n,1}$ and $d_{m,n,1}$,
respectively.  Hence, in the cyclic word $\mathbf{c}$ the occurrence of $c_{m,n,1}$ (or $d_{m,n,1}$) forces the entire
backward chain
\[
d_{m-n,0,1},\,c_{m-n+1,1,1},\,d_{m-n+1,1,1},\,\dots,\,c_{m,n,1}\,(,\,d_{m,n,1})
\]
to appear consecutively in $\mathbf{c}$ (up to cyclic permutation).  In particular, $\mathbf{c}$ contains at least $n$ distinct steps,
so $n\le N$.

If, on the other hand, $\mathbf{c}$ contains no $c_{\ast,\ast,1}$ and no $d_{\ast,\ast,1}$, then $\mathbf{c}$ is supported entirely on
the set $\{c_{m,n,2},c_{m,n,3},d_{m,n,2},d_{m,n,3}\}$.  Inspecting the transition table shows that along any admissible
step within this set the coordinate $n$ never increases, and it strictly decreases whenever one leaves a chamber
$c_{m,n,3}$ with $n\neq 0$ (since $c_{m,n,3}\to d_{m-1,n-1,2}$ for $n\neq 0$).  Therefore, a closed gallery can exist
only if $n=0$ throughout.  In particular, again $n\le N$.

Thus every chamber in $\mathbf{c}$ has $n\le N$.

\medskip
\noindent\textbf{Step 2: a uniform bound on the difference $m-n$.}
Set $r=m-n\ge 0$.  We claim that every chamber occurring in $\mathbf{c}$ satisfies $r\le N$.

Assume $r_{\max}:=\max\{m-n : c_{m,n,i}\text{ or }d_{m,n,i}\text{ occurs in }\mathbf{c}\}>0$.
Since $\mathbf{c}$ is closed, it must contain a step that decreases the value of $m-n$ at some point; inspecting the weight
table, the \emph{only} admissible transition that decreases $m-n$ is
\[
c_{m,n,2}\longrightarrow d_{m-1,n,3}\qquad(m\neq n),
\]
which lowers $m-n$ by $1$.  Hence there exists an occurrence of some $c_{m,n,2}$ in $\mathbf{c}$ with $m-n=r_{\max}$ and
$m\neq n$.  For such an occurrence, the next step is forced:
\[
x_j=c_{m,n,2}\ \Rightarrow\ x_{j+1}=d_{m-1,n,3},
\]
and then the following step is also forced:
\[
d_{m-1,n,3}\longrightarrow c_{m-1,n,2}.
\]
Iterating this forced two-step pattern, we obtain that $\mathbf{c}$ contains (as a consecutive subgallery, up to cyclic shift)
the zig-zag
\[
c_{m,n,2},\,d_{m-1,n,3},\,c_{m-1,n,2},\,d_{m-2,n,3},\,\dots,\,c_{n+1,n,2},\,d_{n,n,3},\,c_{n,n,2},
\]
whose length is $2(m-n)+1=2r_{\max}+1$.  Therefore $N\ge 2r_{\max}+1$, in particular $r_{\max}\le N$.

Consequently, every chamber in $\mathbf{c}$ satisfies $m-n\le N$.

\medskip
\noindent By Steps~1--2, any length-$N$ admissible closed gallery is supported on the finite set of pointed chambers
\[
\{\,c_{m,n,i},\,d_{m,n,i}\ :\ 0\le n\le N,\ 0\le m-n\le N,\ i\in\{1,2,3\}\,\},
\]
which is finite.  Hence there are only finitely many type $1$ admissible closed gallery classes of length $N$ in $X$.
\end{proof}

\noindent
\textbf{Remark.}
In the non-uniform quotient $\Gamma\backslash \mathcal{B}$, a single adjacency step may admit
several distinct tailless lifts in the building $\mathcal{B}$.
The weight of an admissible gallery records precisely this multiplicity and is therefore
intrinsic to the non-uniform setting.
With this convention, the trace of $T^n$ coincides with the total weight of closed
admissible galleries of length $n$, explaining the trace--gallery correspondence in
Lemma~\ref{lem:4.1}.

\section{Chamber zeta function and a determinant formula}\label{sec:4}

In this section, we define the type~$1$ chamber zeta function for $\Gamma\backslash\mathcal{B}$
and establish a determinant expression in terms of the chamber transfer operator. Throughout the section, unless stated otherwise, ``type~$1$ closed galleries'' means closed pointed chamber
galleries whose every step is of step type~$1$ in the sense of Definition~\ref{def:type-k-pointed-gallery-pgl3}.

\subsection{Weights, zeta function, and the transfer operator}

Let $[\mathbf{c}]=(c_0,\ldots,c_{n-1},c_n)$ be a type~$1$ closed gallery class in $\Gamma\backslash\mathcal{B}$.
Its weight is defined by
\[
w([\mathbf{c}])=\prod_{j \,\mathrm{mod}\, n} w(c_j,c_{j+1}).
\]
The (weighted) type~$1$ chamber zeta function is
\[
Z_{\Gamma}(u)=\prod_{[\mathbf{p}]}
\left(1-w([\mathbf{p}])u^{\ell([\mathbf{p}])}\right)^{-1},
\]
where the product runs over all primitive type~$1$ closed admissible gallery classes,
and $\ell([\mathbf{p}])$ denotes the length of the class.

Let $\mathrm{C}(\Gamma\backslash\mathcal{B})$ be the set of all type~$1$ pointed chambers in $\Gamma\backslash\mathcal{B}$.
We consider the complex vector space of finitely supported formal linear combinations
\[
S\bigl(\mathrm{C}(\Gamma\backslash\mathcal{B})\bigr)
:=
\bigoplus_{c\in \mathrm{C}(\Gamma\backslash\mathcal{B})}\mathbb{C}c,
\]
equipped with the inner product
\[
\langle v,w\rangle := \sum_{c} v_c \overline{w_c}
\qquad
\text{for } v=\sum_c v_c c,\; w=\sum_c w_c c.
\]
(The sum is finite because $v$ and $w$ are finitely supported.)

Define the chamber transfer operator
\[
T: S\bigl(\mathrm{C}(\Gamma\backslash\mathcal{B})\bigr)\to
S\bigl(\mathrm{C}(\Gamma\backslash\mathcal{B})\bigr),
\qquad
Tc := \sum_{c'} w(c,c')\,c',
\]
where the sum runs over all type~$1$ pointed chambers $c'$ that are edge-adjacent to $c$.

\subsection{Traceability and the trace formula}

\begin{definition}\label{def:traceable}
Let $A$ be a linear operator on $S(\mathrm{C}(\Gamma\backslash\mathcal{B}))$.
We say that $A$ is \emph{traceable} if the set
\[
\mathrm{Fix}(A):=\{c\in \mathrm{C}(\Gamma\backslash\mathcal{B}) : \langle Ac,c\rangle \neq 0\}
\]
is finite. In that case we define
\[
\operatorname{Tr}(A) := \sum_{c\in \mathrm{Fix}(A)} \langle Ac,c\rangle.
\]
Similarly, if $A^n$ is traceable, we define $\operatorname{Tr}(A^n)$ in the same way.
\end{definition}

\begin{lem}\label{lem:4.1}
For every $n\in\mathbb{N}$, the operator $T^n$ is traceable, and
\[
\operatorname{Tr}(T^n)
=
\sum_{[\mathbf{c}]:\,\ell([\mathbf{c}])=n} w([\mathbf{c}])\,\ell([\mathbf{c}]_0),
\]
where $[\mathbf{c}]$ runs over all type~$1$ closed admissible gallery classes of length $n$,
and $[\mathbf{c}]_0$ denotes the unique primitive class such that $[\mathbf{c}]=[\mathbf{c}]_0^k$.
\end{lem}

\begin{proof}
By Lemma~\ref{lem:3.1}, for each fixed length $n$ there are only finitely many closed admissible
gallery classes of length $n$. In particular, only finitely many type~$1$ pointed chambers $c$
can lie on some closed admissible gallery of length $n$, hence only finitely many $c$ satisfy
$\langle T^n c,c\rangle\neq 0$. This proves that $T^n$ is traceable.

For a fixed $c$, the coefficient of $c$ in $T^n c$ is the sum of the weights of all length-$n$
closed admissible galleries that start and end at $c$, i.e. the sum of $w([\mathbf{c}])$ over all
classes $[\mathbf{c}]$ of length $n$ containing $c$. Summing $\langle T^n c,c\rangle$ over all $c$,
each class $[\mathbf{c}]=[\mathbf{c}]_0^k$ contributes exactly $\ell([\mathbf{c}]_0)$ times, once for
each cyclic shift along the underlying primitive loop. Therefore,
\[
\operatorname{Tr}(T^n)
=
\sum_{[\mathbf{c}]:\,\ell([\mathbf{c}])=n} w([\mathbf{c}])\,\ell([\mathbf{c}]_0),
\]
as claimed.
\end{proof}

\begin{prop}\label{prop:4.2}
The type~$1$ chamber zeta function is formally described by
\[
Z_{\Gamma}(u)=
\exp\left(\sum_{n=1}^{\infty}\frac{\operatorname{Tr}(T^n)}{n}u^n\right)
\in \mathbb{C}[\![u]\!].
\]
\end{prop}

\begin{proof}
As in \cite{DK18,HK3}, taking logarithms and expanding $\log(1-x)$ as a formal power series gives
\[
Z_{\Gamma}(u)^{-1}
=
\prod_{[\mathbf{c}]_0}\left(1-w([\mathbf{c}]_0)u^{\ell([\mathbf{c}]_0)}\right)
=
\exp\left(-\sum_{[\mathbf{c}]_0}\sum_{m\ge1}\frac{w([\mathbf{c}]_0)^m}{m}u^{m\ell([\mathbf{c}]_0)}\right).
\]
Re-indexing the sum over all closed classes $[\mathbf{c}]=[\mathbf{c}]_0^m$ yields
\[
Z_{\Gamma}(u)^{-1}
=
\exp\left(-\sum_{n\ge1}\frac{u^n}{n}
\sum_{[\mathbf{c}]:\,\ell([\mathbf{c}])=n} w([\mathbf{c}])\,\ell([\mathbf{c}]_0)\right)
=
\exp\left(-\sum_{n\ge1}\frac{\operatorname{Tr}(T^n)}{n}u^n\right),
\]
and the claim follows.
\end{proof}

\subsection{Determinant definition}

\begin{definition}\label{def:determinant}
We define the (formal) determinant of $I-uT$ by
\[
\det(I-uT)
:=
\exp\left(-\sum_{n=1}^{\infty}\frac{\operatorname{Tr}(T^n)}{n}u^n\right)
\in 1+u\,\mathbb{C}[\![u]\!].
\]
\end{definition}

\begin{prop}[Determinant formula]\label{prop:det-formula}
With the determinant defined in Definition~\ref{def:determinant}, we have the identity
\[
Z_{\Gamma}(u)=\det(I-uT)^{-1}
\qquad\text{in }\mathbb{C}[\![u]\!].
\]
\end{prop}

\begin{proof}
This is immediate from Proposition~\ref{prop:4.2} and Definition~\ref{def:determinant}.
\end{proof}

\begin{coro}[Rationality and meromorphic continuation]\label{coro:rationality}
For $\Gamma=\mathrm{PGL}_3(\mathbb{F}_q[t])$ acting on the Bruhat--Tits building $\mathcal{B}$
of $G=\mathrm{PGL}_3(\mathbb{F}_q(\!(t^{-1})\!))$, the type~$1$ chamber zeta function
\[
Z_{\Gamma}(u)=\prod_{[\mathbf{c}]}\left(1-w([\mathbf{c}])u^{\ell([\mathbf{c}])}\right)^{-1}
\]
defines a well-defined formal power series in $u$ with constant term $1$, and it satisfies
\[
Z_{\Gamma}(u)=\det(I-uT)^{-1}
\qquad\text{in }\mathbb{C}[\![u]\!].
\]
Moreover, in Section~5 we compute $\det(I-uT)$ explicitly (equivalently, $Z_\Gamma(u)$),
and in particular $Z_{\Gamma}(u)$ extends to a rational function of $u$, hence admits
meromorphic continuation to $\mathbb{P}^1(\mathbb{C})$.
\end{coro}

\subsection{Bridge to Section~\ref{sec:5}: truncations and finite matrix determinants}\label{subsec:bridge_45}

In Section~\ref{sec:5}, we compute $\det(I-uT)$ (equivalently, $Z_\Gamma(u)$) by approximating the
infinite quotient $X=\Gamma\backslash\mathcal{B}$ with finite truncations and passing to a stabilized limit.
We briefly record the precise truncation scheme used there and explain why it captures the determinant of Section~\ref{sec:4}.

For $k\ge 0$, let $X_k$ be the subcomplex of $X$ consisting of vertices $v_{m,n}$ with $n\le k$,
and let $T_k$ denote the restriction of the transfer operator $T$ to $X_k$ as in Section~\ref{sec:5}.
For $N\ge 0$, let $X_{k,N}$ be the finite subcomplex of $X_k$ consisting of vertices $v_{m,n}$ with $m\le N$.
Let $M_{k,N}(u)$ be the matrix representation of the restriction of $I-uT_k$ to $X_{k,N}$
(constructed explicitly in Section~\ref{sec:5}).

\begin{prop}\label{prop:bridge_traces_to_matrices}
For each $n\ge 1$, there exist integers $k(n)$ and $N(n)$ such that for all $k\ge k(n)$ and $N\ge N(n)$,
the coefficient of $u^n$ in $\det(M_{k,N}(u))^{-1}$ equals the coefficient of $u^n$ in $Z_\Gamma(u)$.
Equivalently, we have coefficientwise stabilization in $\mathbb{C}[\![u]\!]$:
\[
Z_\Gamma(u)
=
\lim_{k\to\infty}\;\lim_{N\to\infty}\det(M_{k,N}(u))^{-1}.
\]
\end{prop}

\begin{proof}[Proof]
By Lemma~\ref{lem:3.1}, for each fixed length $n$ there are only finitely many closed admissible gallery classes of
length $n$ in $X$. In particular, all chambers appearing in such galleries lie in a finite region of $X$:
there exist bounds $k(n)$ and $N(n)$ such that every closed admissible gallery of length $n$ is contained in $X_{k(n),N(n)}$.
Hence the weighted counts of length-$n$ closed galleries (equivalently, $\operatorname{Tr}(T^n)$ from Lemma~\ref{lem:4.1})
are the same as those obtained from the truncated dynamics inside $X_{k,N}$ once $k\ge k(n)$ and $N\ge N(n)$.
On the other hand, for each fixed $(k,N)$ the complex $X_{k,N}$ is finite, so $\det(M_{k,N}(u))$ is a genuine polynomial
and the standard finite-dimensional identity expresses $\det(M_{k,N}(u))^{-1}$ as the exponential generating series of
the traces of the corresponding truncated powers. Comparing coefficients yields the stated stabilization.
\end{proof}

In Section~\ref{sec:5}, we compute $\det(M_{k,N}(u))$ by a block-matrix reduction (Schur complement),
show that the limit as $N\to\infty$ stabilizes at a finite matrix $A_{k,0}(u)$, and finally pass to the limit $k\to\infty$.
This produces the explicit closed formula for $\det(I-uT)$ and completes the proof of Theorem~\ref{thm:1.1}.

\section{Proof of Theorem \ref{thm:1.1}}\label{sec:5}
In this section, we prove the main theorem of this paper. The method for the proof is similar to \cite{HK3}. We compute the chamber zeta function of certain truncated subcomplexes and take the limit.

By Proposition~\ref{prop:bridge_traces_to_matrices}, the coefficients of $Z_\Gamma(u)$ are recovered
from the finite truncations $X_{k,N}$ by coefficientwise stabilization:
for each fixed degree $n$, the coefficient of $u^n$ is already determined by $\det(M_{k,N}(u))^{-1}$
once $k$ and $N$ are sufficiently large (depending on $n$).
Thus it suffices to compute $\det(M_{k,N}(u))$ explicitly and then identify its stabilized limits as $N\to\infty$ and $k\to\infty$.

\subsection{The determinant of $I-uT_k$} Let $X_k$ be the subcomplex of $X$ consisting of vertices with $n\leq k$. The restriction $T_k$ of $T$ to $X_k$ is defined by 
\begin{align*}
T_k c_{m,n,1}&=(q-1)c_{m,n,2}+d_{m,n,1}\\
T_k c_{m,n,2}&=
\begin{cases}
qc_{m,n,3} & \text{if } m=n,\\
qd_{m-1,n,3} & \text{if } m\neq n
\end{cases}\\
T_k c_{m,n,3}&=
\begin{cases}
qc_{m,n,1} & \text{if } n=0,\\
qd_{m-1,n-1,2} & \text{if } n\neq 0
\end{cases}\\
T_k d_{m,n,1}&=
\begin{cases}
c_{m+1,n+1,1}+(q-1)d_{m,n,2} & \text{if } n<k-1,\\
(q-1)d_{m,n,2} & \text{if } n=k-1
\end{cases}\\
T_k d_{m,n,2}&=c_{m+1,n,3}+(q-1)d_{m,n,3}\\
T_k d_{m,n,3}&=qc_{m,n,2}.
\end{align*}
where $c_{m,n,i}$ and $d_{m,n,i}$ are the type 1 pointed chambers in Section \ref{sec:3.2}.

To obtain the matrix representation of a subcomplex of $X_k$, we define $6\times 6$ matrices $a_1,a_2,a_3,a_4,b,c,d,e$ as follows:

\begin{align*}
&\begin{aligned}
(a_1(u))_{ij} &=
\begin{cases}
-qu & \text{if } (i,j)=(1,3),\\
(a_3(u))_{ij} & \text{otherwise,}
\end{cases}\\
(a_3(u))_{ij} &=
\begin{cases}
-qu & \text{if } (i,j)=(3,2),\\
(a_4(u))_{ij} & \text{otherwise,}
\end{cases}
\end{aligned}
\qquad
\begin{aligned}
(a_2(u))_{ij} &=
\begin{cases}
-qu & \text{if } (i,j)=(1,3),\\
(a_4(u))_{ij} & \text{otherwise.}
\end{cases}
\end{aligned}
\\[1.2ex]
&(a_4(u))_{ij}=
\begin{cases}
1 & \text{if } i=j,\\
-u & \text{if } (i,j)=(4,1),\\
-(q-1)u & \text{if } (i,j)=(2,1),(5,4),(6,5),\\
-qu & \text{if } (i,j)=(2,6),\\
0 & \text{otherwise,}
\end{cases}
\\[1.2ex]
&\begin{aligned}
(b(u))_{ij} &=
\begin{cases}
-qu & \text{if } (i,j)=(5,3),\\
0 & \text{otherwise,}
\end{cases}\\
(d(u))_{ij} &=
\begin{cases}
-qu & \text{if } (i,j)=(6,2),\\
0 & \text{otherwise.}
\end{cases}
\end{aligned}
\qquad
\begin{aligned}
(c(u))_{ij} &=
\begin{cases}
-u & \text{if } (i,j)=(1,4),\\
0 & \text{otherwise,}
\end{cases}\\
(e(u))_{ij} &=
\begin{cases}
-u & \text{if } (i,j)=(3,5),\\
0 & \text{otherwise.}
\end{cases}
\end{aligned}
\end{align*}

Using the matrices $a_1,a_2,a_3,a_4,b,c,d$ and $e$, we define $k\times k$ block matrices $A_k,B_k,C_k,D_k$ as follows:
\begin{equation*}
\begin{split}
(A_k(u))_{ij}&=\begin{cases}a_1(u)&\text{ if }i=j=1\\
a_2(u)&\text{ if }i=j>1\\
b(u)&\text{ if }j=i+1\\
c(u)&\text{ if }i=j+1\\
0&\text{ otherwise}
\end{cases}\\
(C_k(u))_{ij}&=\begin{cases}d(u)&\text{ if }i=j\\
0&\text{ otherwise}
\end{cases}
\end{split}
\quad\quad
\begin{split}
(B_k(u))_{ij}&=\begin{cases}a_3(u)&\text{ if }i=j=1\\
a_4(u)&\text{ if }i=j>1\\
b(u)&\text{ if }j=i+1\\
c(u)&\text{ if }i=j+1\\
0&\text{ otherwise}
\end{cases}\\
(D_k(u))_{ij}&=\begin{cases}e(u)&\text{ if }i=j\\
0&\text{ otherwise}.
\end{cases}
\end{split}
\end{equation*}
More concretely, it can be expressed as

\[
A_k(u)=
\begin{pmatrix}
a_1(u) & b(u)   &        &        &  \\
c(u)   & a_2(u) & b(u)   &        &  \\
       & c(u)   & a_2(u) & \ddots &  \\
       &        & \ddots & \ddots & b(u) \\
       &        &        & c(u)   & a_2(u)
\end{pmatrix}
\quad\text{(a $k\times k$ block matrix with $6\times 6$ blocks).}
\]

\[
B_k(u)=
\begin{pmatrix}
a_3(u) & b(u)   &        &        &  \\
c(u)   & a_4(u) & b(u)   &        &  \\
       & c(u)   & a_4(u) & \ddots &  \\
       &        & \ddots & \ddots & b(u) \\
       &        &        & c(u)   & a_4(u)
\end{pmatrix}
\quad\text{(a $k\times k$ block matrix with $6\times 6$ blocks).}
\]

\[
C_k(u)=
\begin{pmatrix}
d(u) &      &      & \\
     & d(u) &      & \\
     &      & \ddots & \\
     &      &      & d(u)
\end{pmatrix} \quad\textrm{ and }\quad D_k(u)=
\begin{pmatrix}
e(u) &      &      & \\
     & e(u) &      & \\
     &      & \ddots & \\
     &      &      & e(u)
\end{pmatrix}.
\]

 Let $X_{k,N}$ be the subcomplex of $X_k$ whose vertices $v_{m,n}$ satisfies $m\leq N.$ The matrix representation $M_{k,N}$ of the restriction of $I-uT_{k}$ to $X_{k,N}$ is given by
\begin{equation*}
M_{k,N}(u)=\begin{cases}
A_k(u)&\text{ if }i=j=1\\
B_k(u)&\text{ if }i=j>1\\
C_k(u)&\text{ if }j=i+1\\
D_k(u)&\text{ if }i=j+1\\
0&\text{ otherwise}.
\end{cases}
\end{equation*}
Let $B_{k,1}(u):=B_k(u)$ and 
$$B_{k,\ell+1}(u)=B_k(u)-C_k(u)B_{k,\ell}(u)^{-1}D_k(u).$$
Then the matrix $B_{k,\ell}(u)$ is of the form
\begin{equation*}
(B_{k,\ell}(u))_{ij}=\begin{cases}a_{k,(s,t),\ell}(u)&\text{ if }(i,j)=(6s,6t-1)\\
B_{k}(u)_{ij}&\text{ otherwise},\end{cases}
\end{equation*}
where
\begin{equation*}
a_{k,(s,t),\ell}=\begin{cases}
-(q-1)u-q^2(q-1)u^4+\displaystyle\sum_{i=1}^k a_{k,(1,i),\ell-1}q^{3}(q-1)u^{2i+4}&\text{if }(s,t)=(1,1)\\
-q^{2}(q-1)u^{2s+2}+\displaystyle\sum_{i=1}^k a_{k,(s,i),\ell-1}q^{3}(q-1)u^{2i+4}&\text{if }s\neq 1,t=1\\
-(q-1)u+a_{k,(s,t-1),\ell-1}q^3u^4&\text{if } s=t>1\\
a_{k,(s,t-1),\ell-1}q^3u^4&\text{if } s\neq t\text{ and }s>1.
\end{cases}
\end{equation*}

By the construction of the sequence of matrices $B_{k,\ell}$, we have
\begin{equation*}
\begin{pmatrix}
I&-C_k(u)B_{k,\ell}(u)^{-1}\\
0&B_{k,\ell}(u)^{-1}
\end{pmatrix}
\begin{pmatrix}
B_{k}(u)&C_k(u)\\
D_{k}(u)&B_{k,\ell}(u)
\end{pmatrix}
=\begin{pmatrix}
B_{k,\ell+1}(u)&0\\
B_{k,\ell}(u)^{-1}D_{k}(u)&I
\end{pmatrix}.
\end{equation*}
Let $A_{k,N}$ be a matrix given by
$$A_{k,N}(u)=A_k-C_k(u)B_{k,N-1}(u)^{-1}D_k(u).$$
The row reduced echelon form $B_{k,\ell}$ is the $6k\times 6k$ identity matrix $I$. This implies that $\det(B_{k,\ell}(u))=1$.
Thus the determinant of $M_{k,N}(u)$ is equal to the determinant of $A_{k,N}(u)$.

The matrix $A_{k,N}(u)$ is given by
\begin{equation*}
(A_{k,N}(u))_{ij}=\begin{cases}a_{k,(s,t),N}(u)&\text{ if }(i,j)=(6s,6t-1)\\
(A_{k}(u))_{ij}&\text{ otherwise}.\end{cases}
\end{equation*}
The determinant of $I-uT_k$ is the limit of the determinant $A_{k,N}(u)$. Let $A_{k,\infty}(u)$ be the limit of $A_{k,N}(u)$ and $a_{k,(s,t)(u)}$ the limit of $a_{k,(s,t),\ell}(u)$. For $t\geq 2$, 
\begin{equation*}
a_{k,{(s,t)}}(u)=\begin{cases}
-(q-1)u+q^3u^4a_{k,(s,t-1)}&\text{if }s=t\\
q^3u^4a_{k,{s,t-1}}&\text{if }s\neq t.
\end{cases}
\end{equation*}
This shows that 
$$\det (A_{k,\infty}(u))=\det( A_{k,0}(u)),$$
where
\begin{equation*}
(A_{k,0}(u))_{ij}=\begin{cases}
a_{k,(s,1)}(u)&\text{if }(i,j)=(6s,5)\\
-q^3u^4&\text{if }(i,j)=(6s-1,6s+5)\\
(A_{k}(u))_{ij}&\text{otherwise}.
\end{cases}
\end{equation*}
Using elementary row operation, the determinant of $A_{k,0}$ is equal to the determinant of the following matrix
\begin{equation*}
\begin{pmatrix}
1                        &0    &-qu&0&0&0\\
-(q-1)u               &1   &0    &0&0&-qu\\
0                        &-qu&1     &0&0&0\\
-u                       &0   &0      &1&0&0\\
-\sum_{i=1}^{2k-1}q^{2i-2}(q-1)u^{3i-1}&0&0&0&1+\sum_{i=2}^{k}q^{4i-5}u^{4i-5}a_{k,(i,1)}(u)&0\\
0&0&0&0&a_{k,(1,1)}(u)&1
\end{pmatrix}.
\end{equation*}
The determinant of $A_{k,0}$ is as follows:
\begin{equation*}
(1-q^2(q-1)u^3)\biggl(1+\sum_{i=2}^{k}q^{4i-5}u^{4i-5}a_{k,(i,1)}(u)\biggr)+q^3u^3a_{k,(1,1)}\sum_{i=1}^{2k-1}q^{2i-2}(q-1)u^{3i-1}.
\end{equation*}

\subsection{The determinant of $I-uT$} For any $t\geq 2$, we obtain
$$a_{k,(1,t)}(u)=q^3u^4a_{k,(1,t-1)}(u)=q^6u^8a_{k,(1,t-2)}(u)=\cdots=q^{3t-3}u^{4t-4}a_{k,(1,1)}(u).$$
Using this, we have
\begin{equation*}
\begin{split}
a_{k,(1,1)}(u)
&=-(q-1)u-q^2(q-1)u^4+a_{k,(1,1)}\sum_{i=1}^k q^{3i}(q-1)u^{6i}
\end{split}
\end{equation*}
and 
\begin{equation*}
a_{k,(1,1)}(u)=\frac{-(q-1)u-q^2(q-1)u^4}{1-\displaystyle\sum_{i=1}^kq^{3i}(q-1)u^{6i}}.
\end{equation*}
The sequence $a_{k,(1,1)}$ converges to
$$a_{(1,1)}=\frac{-(1-q^3u^6)((q-1)u+q^2(q-1)u^4)}{1-q^4u^6}.$$
For any $s,t,$ with $s\geq 2$ and $t< s,$
$$a_{k,(s,t)}(u)=q^3u^4a_{k,(s,t-1)}(u)=q^6u^8a_{k,(1,t-2)}(u)=\cdots=q^{3t-3}u^{4t-4}a_{k,(s,1)}(u).$$
For any $s,t$ with $s\geq 2$ and $t\geq s$,
\begin{equation*}
\begin{split}
a_{k,(s,t)}(u)&=q^3u^4a_{k,(s,t-1)}(u)=\cdots=q^{3(t-s)}u^{4(t-s)}a_{k,(s,s)}(u)\\
&=-q^{3(t-s)}(q-1)u^{4(t-s)+1}+q^{3(t-s+1)}u^{4(t-s+1)}a_{k,(s,s-1)}(u)\\
&=-q^{3(t-s)}(q-1)u^{4(t-s)+1}+q^{3(t-1)}u^{4(t-1)}a_{k,(s,1)}(u).
\end{split}
\end{equation*}
Similarly, 
\begin{equation*}
\begin{split}
a_{k,(s,1)}=&-q^{2}(q-1)u^{2s+2}+\displaystyle\sum_{i=1}^k a_{k,(s,i)}q^{3}(q-1)u^{2i+4}\\
=&-q^{2}(q-1)u^{2s+2}+\displaystyle\sum_{i=1}^k a_{k,(s,1)}q^{3i}(q-1)u^{6i}\\
&-\sum_{i=s}^kq^{3(i-s+1)}(q-1)^2u^{6i-4s+5}
\end{split}
\end{equation*}
and 
\begin{equation*}
a_{k,(s,1)}=\frac{-q^{2}(q-1)u^{2s+2}-\displaystyle\sum_{s}^kq^{3(i-s+1)}(q-1)u^{6i-4s+5}}{1-\displaystyle\sum_{i=1}^k q^{3i}(q-1)^2u^{6i}}.
\end{equation*}
The sequence $a_{k,(s,1)}$ converges to 
\begin{equation*}
a_{(s,1)}=\frac{-q^2(q-1)u^{2s+2}(1-q^3u^6)-q^3(q-1)^2u^{2s+5}}{1-q^4u^6}.
\end{equation*}
Using the limit of $\det(A_{k,0})$, the determinant of $I-uT$ is 
\begin{equation*}
\lim_{k\rightarrow \infty}\det(A_{k,0})=\frac{(1-q^3u^6)(1-q^3u^3)}{(1-q^4u^6)(1-q^2u^3)}.
\end{equation*}
By determinant formula, we obtain the following theorem:
\begin{thm}\label{thm:5.1} Let $\Gamma=\PGL_3(\mathbb{F}_q[t])$. The type 1 chamber zeta function $Z_\Gamma(u)$ converges for sufficiently small $u$ and it is given by
$$Z_{\Gamma}(u)=\frac{(1-q^4u^6)(1-q^2u^3)}{(1-q^3u^6)(1-q^3u^3)}.$$
\end{thm}
Recall that we defined $$N_n(\Gamma\backslash\mathcal{B})=\sum_{c\colon\ell([\mathbf{c}])=n}w([\mathbf{c}])\ell([\mathbf{c}]_0)$$ by the weighted number of type $1$ closed galleries in $\Gamma\backslash \mathcal{B}$ of length $n$.
The trace of the operator $T^m$ coincides with $N_m(\Gamma\backslash\mathcal{B})$. By Proposition \ref{prop:4.2}, we have
\begin{equation}\label{eq:5.1}
u\frac{d}{du}\log Z_{\Gamma}(u)=u\frac{Z'_{\Gamma}(u)}{Z_{\Gamma}(u)}=\sum_{m=1}^{\infty} N_m(\Gamma\backslash \mathcal{B})u^m.
\end{equation}
\begin{coro} Let $\Gamma=\operatorname{PGL}(3,\mathbb{F}_q[t])$ and $N_m(\Gamma\backslash\mathcal{B})$ as above. Then, we have
$$N_m(\Gamma\backslash\mathcal{B})=\left\{\begin{array}{ll}3q^{3r}-3q^{2r} & \textrm{ if }m=3r\textrm{ and }m\not\equiv 0 \,(\operatorname{mod}\,{6}) \\ 
3q^{6r}-9q^{4r}+6q^{3r} & \textrm{ if }m=6r \\ 0 
& \textrm{ otherwise}\end{array}\right..$$
\end{coro}
\begin{proof}
By Theorem~5.1, we have
\[
Z_\Gamma(u)
=
\frac{(1-q^{4}u^{6})(1-q^{2}u^{3})}{(1-q^{3}u^{6})(1-q^{3}u^{3})}.
\]
Recall from \eqref{eq:5.1} that
\[
u\frac{d}{du}\log Z_\Gamma(u)
=
\sum_{m\ge1} N_m(\Gamma\backslash B)\,u^m.
\]
Taking the logarithmic derivative of $Z_\Gamma(u)$, we obtain
\begin{align*}
u\frac{d}{du}\log Z_\Gamma(u)
&=
u\frac{d}{du}\Bigl(
\log(1-q^{4}u^{6})
+
\log(1-q^{2}u^{3})
-
\log(1-q^{3}u^{6})
-
\log(1-q^{3}u^{3})
\Bigr) \\
&=
-\frac{6q^{4}u^{6}}{1-q^{4}u^{6}}
-\frac{3q^{2}u^{3}}{1-q^{2}u^{3}}
+\frac{6q^{3}u^{6}}{1-q^{3}u^{6}}
+\frac{3q^{3}u^{3}}{1-q^{3}u^{3}}.
\end{align*}
Equivalently,
\begin{equation}\label{eq:Nm-generating}
\sum_{m\ge1} N_m(\Gamma\backslash B)\,u^m
=
\frac{3q^{3}u^{3}}{1-q^{3}u^{3}}
-
\frac{3q^{2}u^{3}}{1-q^{2}u^{3}}
+
\frac{6q^{3}u^{6}}{1-q^{3}u^{6}}
-
\frac{6q^{4}u^{6}}{1-q^{4}u^{6}}.
\end{equation}
Each term on the right-hand side admits a geometric series expansion:
\[
\frac{3q^{3}u^{3}}{1-q^{3}u^{3}}
=
3\sum_{r\ge1} q^{3r}u^{3r},
\qquad
\frac{3q^{2}u^{3}}{1-q^{2}u^{3}}
=
3\sum_{r\ge1} q^{2r}u^{3r},
\]
\[
\frac{6q^{3}u^{6}}{1-q^{3}u^{6}}
=
6\sum_{r\ge1} q^{3r}u^{6r},
\qquad
\frac{6q^{4}u^{6}}{1-q^{4}u^{6}}
=
6\sum_{r\ge1} q^{4r}u^{6r}.
\]
It follows immediately that $N_m(\Gamma\backslash B)=0$ unless $3\,|\, m$.
If $m=3r$ and $m\not\equiv 0 \pmod{6}$, then only the first two series contribute,
and we obtain
\[
N_{3r}(\Gamma\backslash B)
=
3q^{3r}-3q^{2r}.
\]
If $m=6r$, then all four series in \eqref{eq:Nm-generating} contribute, yielding
\[
N_{6r}(\Gamma\backslash B)
=
\bigl(3q^{6r}-3q^{4r}\bigr)
+
\bigl(6q^{3r}-6q^{4r}\bigr)
=
3q^{6r}-9q^{4r}+6q^{3r}.
\]
This completes the proof.
\end{proof}

\begin{remark}[Hecke operators and (co)homology]
The chamber transfer operator $T$ may also be viewed as a specific element of the Hecke algebra of
\(
G=\mathrm{PGL}_3\bigl(\mathbb{F}_q(\!(t^{-1})\!)\bigr),
\)
acting on spaces of $\Gamma$-invariant functions on the building $\mathcal{B}$. In particular, $T$ acts naturally on various homology and cohomology groups of $\Gamma\backslash\mathcal{B}$ with coefficients in local systems arising from finite-dimensional representations of $G$. It would be interesting to investigate to what extent the chamber zeta function $Z_{\Gamma}(u)$ can be factored into contributions coming from Hecke operators acting on such (co)homology groups, in analogy with the factorization of Selberg zeta functions into automorphic $L$-functions in the archimedean setting. We expect that this perspective should be especially relevant for higher-rank groups and for quotients carrying additional arithmetic structure.
\end{remark}

\end{document}